\documentclass[reqno]{amsart}%

 \pdfoutput=1

\usepackage{palatino}

\usepackage{amsfonts}
\usepackage{amsmath}
\usepackage{amssymb}

\usepackage{latexsym,wasysym,mathrsfs,bbm,stmaryrd}

\usepackage{graphicx}%
\providecommand{\U}[1]{\protect\rule{.1in}{.1in}}
\newtheorem{theorem}{Theorem}[section]

\newtheorem{corollary}[theorem]{Corollary}

\newtheorem{definition}[theorem]{Definition}

\newtheorem{lemma}[theorem]{Lemma}

\newtheorem{proposition}[theorem]{Proposition}

\theoremstyle{remark}
\newtheorem{remark}[theorem]{Remark}
\numberwithin{equation}{section}

\begin{document}

\title{Brown measure support and the free multiplicative Brownian motion}
\author{Brian C.\ Hall}
\address{Department of Mathematics \\
University of Notre Dame \\
Notre Dame, IN  46556}
\email{bhall@nd.edu}
\author{Todd Kemp}
\thanks{Kemp's research is supported in part by NSF CAREER Award DMS-1254807 and NSF Grant DMS-1800733}
\address{Department of Mathematics\\
University of California, San Diego \\
La Jolla, CA 92093-0112}
\email{tkemp@ucsd.edu}

\begin{abstract}
The free multiplicative Brownian motion $b_{t}$ is the large-$N$
limit of Brownian motion $B_t^N$ on the general linear group $\mathrm{GL}(N;\mathbb{C})$. We
prove that the Brown measure for $b_{t}$---which is an analog of the
empirical eigenvalue distribution for matrices---is supported on the closure
of a certain domain $\Sigma_{t}$ in the plane. The domain $\Sigma_t$ was introduced by Biane in the context
of the large-$N$ limit of the Segal--Bargmann transform associated to $\mathrm{GL}(N;\mathbb{C})$.

We also consider a two-parameter version, $b_{s,t}$: the large-$N$ limit of a related family of diffusion
processes on $\mathrm{GL}(N;\mathbb{C})$ introduced by the second author.  We show that the Brown
measure of $b_{s,t}$ is supported on the closure of a certain planar domain $\Sigma_{s,t}$,
generalizing $\Sigma_t$, introduced by Ho.

In the process, we introduce a new family of spectral domains related to any operator in a tracial von Neumann
algebra: the {\em $L^p_n$-spectrum} for $n\in\mathbb{N}$ and $p\ge 1$, a subset of the ordinary spectrum
defined relative to potentially-unbounded inverses.  We show that, in general, the support of the Brown
measure of an operator is contained in its $L_2^2$-spectrum.

\end{abstract}

\maketitle

\tableofcontents

\section{Introduction}

One of the core theorems in random matrix theory is the {\em circular law}. Suppose $C^N$ is an $N\times N$
complex matrix whose entries are independent centered normal random variables of variance
$\frac1N$.  Then the empirical eigenvalue distribution of $C^N$ (the random probability measure placing
points of equal mass at the eigenvalues) converges almost surely to the uniform probability measure on the unit disk
as $N\to\infty$. This theorem is due to Ginibre \cite{Gin} and $C^N$ is often
called a {\em Ginibre ensemble}. The circular law has been incrementally generalized to its strongest
form where the entries are independent but can have any distribution with two finite moments \cite{Bai,Girko,GT,TV}.

We can recast this as a theorem about {\em matrix-valued Brownian motion}.  In a finite-dimensional real Hilbert space,
there is a canonical Brownian motion, constructed by adding independent standard real Brownian
motions in all the directions of any orthonormal basis.  (See Section \ref{section Lie group BM} below.)
Let us regard the space $M_N(\mathbb{C})$ of complex $N\times N$ matrices as a real vector space of dimension $2N^2$, and equip it with the inner product
\begin{equation} \label{e.scaled.inner.prod}
\left\langle X,Y\right\rangle _{N}=N\operatorname{Re}\mathrm{Trace}(X^{\ast}Y).
\end{equation}
Then the associated Brownian motion $C^N_t$ has the same law as the Ginibre ensemble, scaled by a factor of $\sqrt{t}$.  (The factor $N$ in front of the Hilbert--Schmidt inner product is the correct choice to give the entries variance of order $1/N$.)
Hence, the matrix-valued Brownian motion $C^N_t$ has eigenvalues that concentrate uniformly in the disk of
radius $\sqrt{t}$ as $N\to\infty$.

In this paper, we are interested in the Brownian motion $B^N_t$ on the general linear group $\mathrm{GL}(N;\mathbb{C})$.  One nice
geometric way to define this object is by the {\em rolling map}. The tangent space to the identity in $\mathrm{GL}(N;\mathbb{C})$
(i.e.,\ the Lie algebra of this Lie group) is all of $M_N(\mathbb{C})$. Take the Brownian paths in the tangent space, and roll
them onto the group; this yields the paths of $B^N_t$.  Since the paths are not smooth this rolling 
is accomplished by a \textit{stochastic} differential equation for $B^N_t$ in terms of $C^N_t$ (cf.\ \eqref{e.BM.Lie.Group}).  In particular, for small 
time, $B^N_t$ and $C^N_t$ are "close", and so it is natural to expect that the eigenvalues of $B^N_t$ should follow a
deformation of the circular law of radius $\sqrt{t}$.

For a normal matrix $A$, the eigenvalues are encoded in the
matrix moments $\{\mathrm{Trace}[A^k(A^\ast)^{\ell}]\}_{k,\ell\in\mathbb{N}}$.  However, the ensemble
$C^N_t$ is almost surely non-normal for any $t>0$; in fact, a stronger statement is true: with probability $1$, $C^N_t$ is non-normal for all $t>0$\ \cite[Proposition 4.15]{KempLargeN}. The lack of normality presents significant hurdles to understanding the limit behavior of its
eigenvalues, whose connection to matrix moments is quite a bit more subtle.  Nevertheless, the $\ast$-moments of $C_t^N$ and $(C_t^N)^\ast$ (i.e.,\ traces of all words in these non-commuting matrices) do have a meaningful
large-$N$ limit: in the language of free probability, the ensemble $C^N_t$ {\em converges in $\ast$-distribution}
to an operator $c_t$, cf.\ \cite{Voiculescu} (see Section \ref{section large-N limits}).

This {\em circular Brownian motion} $c_t$, living in noncommutative probability space,
does not have eigenvalues, and is not normal, so it does not have a spectral resolution. 
Nevertheless, there is a construction, known as the \textbf{Brown measure} that reproduces the spectral distribution in the normal case but is also valid for non-normal operators. Each operator $a$ in a tracial von Neumann algebra has an associated Brown measure $\mu_a$, which is
a probability measure supported in the spectrum of $a$ in $\mathbb{C}$.  If $a$ is normal, $\mu_a$ is the usual
spectral measure inherited from the spectral theorem; if $A$ is an $N\times N$ matrix, its Brown measure is simply its empirical eigenvalue distribution.  We discuss the Brown
measure in general in Section \ref{section Brown measure} below.  Girko's proof \cite{Girko} of the general circular law
begins by proving that the Brown measure of $c_1$ is uniform on the unit disk, and then shows
that the empirical eigenvalue distribution of $C^N$ actually converges to the Brown measure of the large-$N$ limit.

Meanwhile, the Brownian motion $B_{t}^{N}$ on the group $\mathrm{GL}(N;\mathbb{C})$
also has a large-$N$ limit in terms of $\ast$-distribution:\ an operator $b_t$ known as the \textbf{free
multiplicative Brownian motion}.  It was introduced by Biane \cite{BianeFields,BianeJFA} and conjectured by him
to be the large-$N$ limit of $B_t^{N}$; this conjecture was proven by the second author in \cite{KempLargeN}.
The first step in understanding the large-$N$ behavior of the eigenvalues of $B_t^N$ is to determine
the Brown measure $\mu_{b_t}$ of $b_t$.  It is a probability measure supported in the spectrum of $b_t$; but $b_t$
is a complicated object, and in particular its spectrum is completely unknown.

In this paper, we identify a closed set $\overline{\Sigma}_t$ (see Section \ref{domains.sec}) which contains the support of the Brown measure $\mu_{b_t}$. The region $\Sigma_t$ was
introduced by Biane in \cite{BianeJFA} in the context of the {\em Segal--Bargmann transform} (or ``Hall transform'')
associated to the unitary group $\mathrm{U}(N)$ and its complexification $\mathrm{GL}(N;\mathbb{C})$
(cf.\ \cite{Ha1994}).  Biane introduced a  {\em free Hall transform} $\mathscr{G}_t$, which he understood as a sort
of large-$N$ limit of the Hall transform for $\mathrm{U}(N)$.  The transform $\mathscr{G}_t$ is an integral operator which
maps functions on the unit circle to a space
$\mathscr{A}_t$ of holomorphic functions on the region $\Sigma_t\subset\mathbb{C}$, whose definition
falls out of the complex analysis used in Biane's proofs.

The meaning of the region $\Sigma_t$ and its relation to the free multiplicative Brownian motion $b_t$ have remained mysterious.
One clue to its origin comes from the holomorphic functional calculus.  Using the metric properties of $\mathscr{G}_t$,
Biane showed that one can make sense of $F(b_t)$, as a possibly unbounded operator, for any $F\in\mathscr{A}_t$.
Now, if the spectrum of $b_t$ were contained in $\Sigma_t$, properties of the holomorphic functional calculus would
show that $F(b_t)$ is a {\em bounded} operator for all $F$ in $\mathcal{H}(\Sigma_t)$ and thus for all $F$ in $\mathscr{A}_t$, which is
(presumably) not the case.  On the other hand, the fact that $F(b_t)$ can be defined at all --- even as an unbounded
operator --- suggests that the spectrum of $b_t$ is at least contained in the {\em closure} of $\Sigma_t$.  Such a
results would then imply that the support of the Brown measure of $b_t$ is contained in $\overline{\Sigma}_t$.  The
latter statement is the main theorem of this paper.

\begin{theorem}
\label{main.thm}For all $t>0$, the support of the Brown measure $\mu_{b_{t}}$
of the free multiplicative Brownian motion $b_{t}$ is contained in
$\overline{\Sigma}_{t}$.
\end{theorem}

We expect that the large-$N$ limit of the empirical eigenvalue
distribution of the Brownian motion $B_{t}^{N}$ on $\mathrm{GL}(N;\mathbb{C})$ 
coincides with the Brown measure $\mu_{b_{t}}$ of the free multiplicative Brownian motion.
If that is the case, the eigenvalues of $B_{t}^{N}$ should concentrate in $\overline{\Sigma}_{t}$
for large $N$; this claim is supported by numerical evidence. Figure \ref{4evalplots.fig}
shows simulations of $B_{t}^{N}$ with $N=2000$ and four different values of $t$, plotted along
with the domains $\Sigma_{t}$. (The domain for $t\ge4$ has a small hole around
the origin, which can be seen in the bottom two images in the figure.)%

\begin{figure}[hptb]%
\centering
\includegraphics[scale=0.61]
{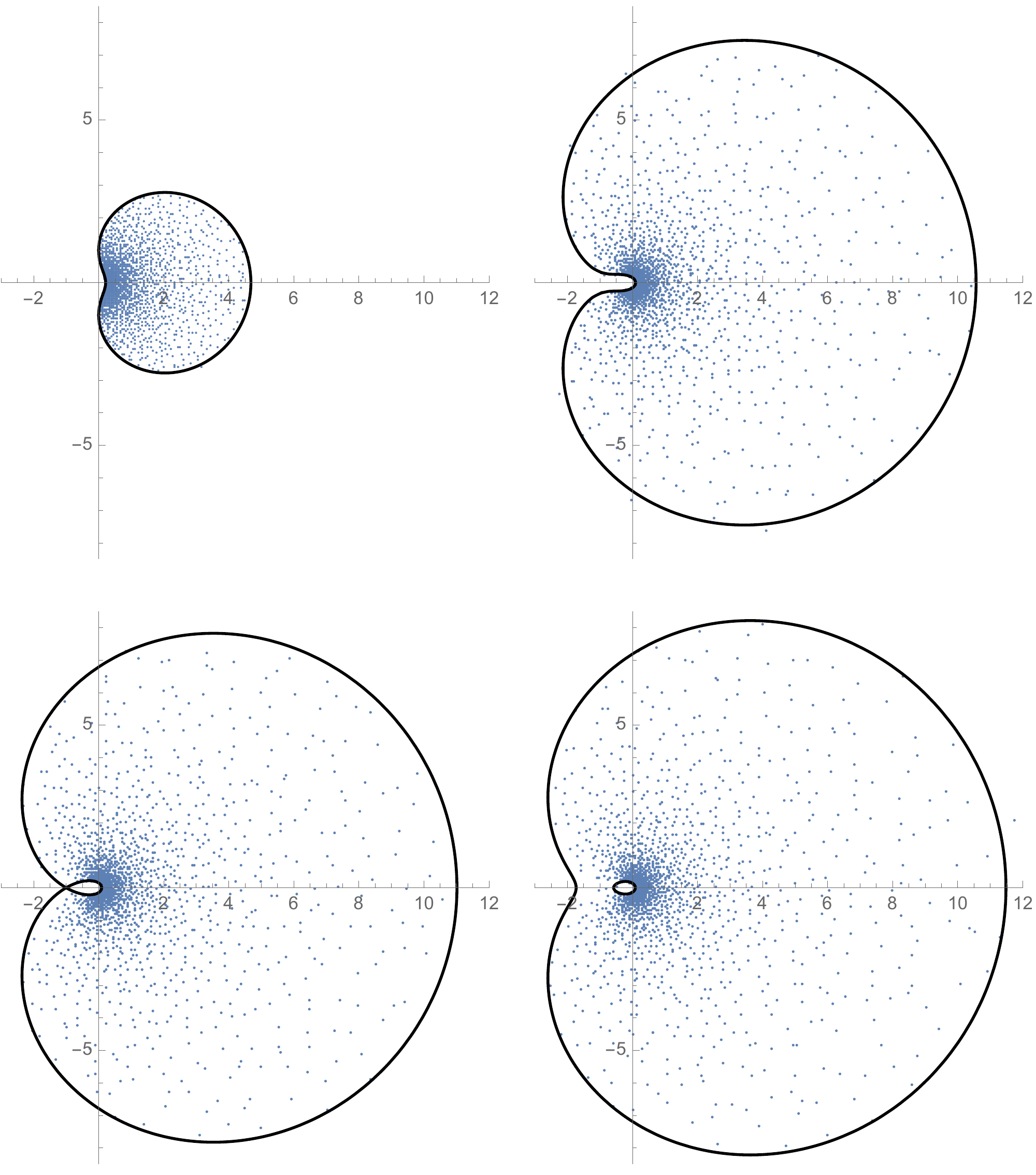}%
\caption{Simulations of the eigenvalues of $B_{t}^{N}$ shown with the domain
$\Sigma_{t}$, for $N=2000$ and $t=2$, $t=3.9$, $t=4$, and $t=4.1$}%
\label{4evalplots.fig}%
\end{figure}

We also consider a two-parameter version $b_{s,t}$ of the free multiplicative
Brownian motion and show that its Brown measure is supported on the closure of
a certain domain $\Sigma_{s,t}$, introduced by Ho \cite{Ho}. These domains similarly arise in
the large-$N$ limit of the two-parameter Segal--Bargmann transform in the Lie group
setting.  The precise statement and proof can be found in Section
\ref{twoParameter.sec}.


After this paper was submitted for publication, we became aware of two
papers in the physics literature that address the large-$N$ limit of
Brownian motion in $GL(N;\mathbb{C}).$ The first, by Gudowska-Nowak, Janik,
Jurkiewicz, Nowak \cite{Nowak} addresses only the one-parameter case (what
we call $B_{t}^{N}$). The second, by Lohmayer, Neuberger, and Wettig \cite%
{Lohmayer} addresses the full two-parameter case (what we call $B_{s,t}^{N}$
in Section 5). Using nonrigorous methods, both papers identify the region in
which the eigenvalues should live in the large-$N$ limit. The domains they
identify are precisely what we call $\Sigma _{t}$ (in the one-parameter case
in \cite{Nowak}) and what we call $\Sigma _{s,t}$ (in the two-parameter case
in \cite{Lohmayer}).

Our results are rigorous and use completely different methods from those in 
\cite{Nowak,Lohmayer}. Our approach also has the conceptual advantage of
connecting the Brown measure of the free multiplicative Brownian motion to
the previously known results on the distribution of the free \textit{unitary}
Brownian motion. Finally, we develop a general result about Brown measures
and the notion of $L_{2}^{2}$ spectrum.  (For more details on the relationship between
\cite{Nowak,Lohmayer} and our results, see Remark \ref{remark.physics} below.)


The strategy we employ to prove Theorem \ref{main.thm} is of independent interest, as it
provides a new restriction on the support of the Brown measure of an operator, in terms of a family of spectral
domains associated to the operator.
Let $(\mathcal{A},\tau)$ be a tracial von Neumann algebra and let $a\in\mathcal{A}$. As noted,
the Brown measure $\mu_a$ is 
supported in the spectrum of $a$; in fact, its support set may be a strict subset of the spectrum.
Recall that the spectrum of $a$ is the complement of the resolvent set of of $a$, which is the set
of $\lambda\in\mathbb{C}$ for which $a-\lambda$ is invertible (meaning that $(a-\lambda)^{-1}$ is a
bounded operator).  The rich structure of $\mathcal{A}$ and the trace $\tau$
give a natural generalization of these spaces: we may ask that the inverse
$(a-\lambda)^{-1}$ exist but not necessarily be bounded, instead insisting that it is in
$L^p(\mathcal{A},\tau)$ for some $p\ge1$. (The $p=\infty$ case coincides with the usual resolvent set.)
What is more, given the often bizarre algebraic properties of non-normal operators (which can, for example,
be nilpotent), we may ask that some power $(a-\lambda)^{n}$ have an inverse in $L^p(\mathcal{A},\tau)$.
(Unless $a$ is normal, this is {\em not} equivalent to $a-\lambda$ having an inverse in $L^{np}(\mathcal{A},\tau)$.)
The set of $\lambda\in\mathbb{C}$ for which $(a-\lambda)^n$ has an inverse in $L^p$ (and for which certain
uniform local bounds hold) is called the {\bf $L^p_n$-resolvent set} of $a$, and its
complement is $\mathrm{spec}^p_n(a)$, the {\bf $L^p_n$-spectrum} of $a$.

Our key observation leading to the proof of Theorem \ref{main.thm} is the following new description of (a closed 
set containing) the support of the Brown measure.

\begin{theorem} \label{thm.L22.spec} For any operator $a$ in a tracial von Neumann algebra, the support of the Brown measure $\mu_a$ is contained in $\mathrm{spec}^2_2(a)$, which is a subset of the spectrum of $a$. \end{theorem}

A detailed discussion of these generalized spectral domains, and the proof of Theorem \ref{thm.L22.spec},
can be found in Section \ref{section L22}.  Once this result is established, we use Biane's
free Hall transform $\mathscr{G}_{t}$ to show that
$\mathrm{spec}_2^2(b_t) = \overline{\Sigma}_t$, which proves that the support of $\mu_{b_t}$
is contained in $\overline{\Sigma}_t$, establishing Theorem \ref{main.thm}.

A more detailed version of this outline of the proof is contained in Section \ref{outline.sec}, (cf.\ Theorems \ref{BrownAndL2Spec.thm} and \ref{L2Inverse.thm}); the complete proofs can then be found in Section \ref{proofs.sec}.

\section{Preliminaries}

In this section, we provide background on the objects in the statement of our
main theorem---the free multiplicative Brownian motion, the Brown measure, and
the domains $\Sigma_{t}$.

\subsection{Lie group Brownian motions and their large-$N$ limits\label{section SDEs}}
\subsubsection{Lie group Brownian motions\label{section Lie group BM}}

Let $H$ be a finite-dimensional real Hilbert space.  The {\em Brownian motion on $H$}, $W^H_t$, is the diffusion
process defined by
\begin{equation} \label{e.BM.Hilbert} W^H_t = \sum_{j=1}^d B^j_t e_j \end{equation}
where $\{e_1,\ldots,e_d\}$ is an orthonormal basis for $H$, and $\{B^j_t\}_{j=1}^d$ are i.i.d.\ standard (real)
Brownian motions.  The law of this process is invariant under rotations, and hence does not depend on which
orthonormal basis is chosen.

Let $G\subset\mathrm{GL}(N;\mathbb C)$ be a matrix Lie group (in $M_N(\mathbb{C})$), and let $\mathfrak{g}\subset M_N(\mathbb{C})$ be its Lie algebra.
A choice of inner product on $\mathfrak{g}$ induces a left-invariant Riemannian metric on $G$.  As a Riemannian manifold, then, $G$ has a
well-defined Brownian motion: the diffusion with infinitesimal generator given by half the Laplacian. In the Lie group setting
there is a simple description of the Brownian motion $B^G_t$ in terms of the Brownian motion $W^{\mathfrak{g}}_t$ on the Lie algebra (as in \eqref{e.BM.Hilbert}):
\begin{equation} \label{e.BM.Lie.Group} dB^G_t = B^G_t\circ dW^{\mathfrak{g}}_t, \qquad B^G_0 = I. \end{equation}
The $\circ$ denotes the Stratonovich stochastic integral.  This Stratonovich SDE can be converted to an It\^o SDE; the form of the resulting equation depends on the structure of the group (cf.\ \cite[p.\ 116]{McKean}).


For our purposes, the two relevant Lie groups are the unitary group $\mathrm{U}(N)$ whose Lie algebra is $\mathfrak{u}(N) = M_N^{\mathrm{s.a.}}(\mathbb{C})$ (self-adjoint matrices), and the general linear group
$\mathrm{GL}(N;\mathbb{C})$ whose Lie algebra consists of all complex matrices $\mathfrak{gl}(N;\mathbb{C}) = M_N(\mathbb{C})$. (In the unitary case, we follow the physicists' convention; mathematicians typically use skew-self-adjoint matrices.)
Using the inner product \eqref{e.scaled.inner.prod} on both these Lie algebras, we obtain  Brownian motions which we will denote thus:
\[ W^{\mathfrak{u}(N)}_t = X^N_t \qquad \text{and} \qquad W^{\mathfrak{gl}(N;\mathbb{C})}_t = C^N_t. \]
To be more explicit, $C^N_t$ is the {\em Ginibre Brownian motion}, which has i.i.d.\ entries that are all complex Brownian motions of variance $t/N$, and $X^N_t$ is the {\em Wigner Brownian motion}, which is Hermitian with i.i.d.\ upper triangular entries, with complex Brownian motions above the diagonal and real Brownian motions on the diagonal, each of variance $t/N$.  

The Brownian motions on the groups, which we denote $B^{\mathrm{U}(N)}_t = U^N_t$ and $B^{\mathrm{GL}(N;\mathbb{C})}_t = B^N_t$, then satisfy Stratonovich SDEs given by \eqref{e.BM.Lie.Group}. These equations can be written in It\^o form as follows:
\begin{equation} \label{e.UB.SDEs} dU^N_t = iU^N_t\,dX^N_t -\textstyle{\frac12}U^N_t\,dt \qquad \text{and} \qquad dB^N_t = B^N_t\,dC^N_t \end{equation}
(both started at the identity matrix).  These defining SDEs play the role of the rolling map described in the introduction.

\subsubsection{The large-$N$ limits\label{section large-N limits}}
The four processes $X^N_t$, $C^N_t$, $U^N_t$, and $B^N_t$ all have large-$N$ limits in the sense of free probability
theory.  (For a thorough introduction to free probability theory and its connection to random matrix theory,
the reader is directed to \cite{MS} and \cite{NS}.) The
limits are one-parameter families of operators $x_t$, $c_t$, $u_t$, and $b_t$ all living in a noncommutative
probability space $(\mathcal{B},\tau)$. (More precisely, $\mathcal{B}$ is a finite von Neumann algebra and $\tau$ is a faithful, normal, tracial state.)
 The sense of convergence is
{\em almost sure convergence of the finite-dimensional noncommutative distributions}, defined as follows.

\begin{definition} \label{def.limit.process} Let $A^N_t$ be a sequence of $M_N(\mathbb{C})$-valued stochastic
processes (all defined on the same sample space).  Let $(\mathcal{A},\tau)$ be a noncommutative probability space,
and let $a_t\in\mathcal{A}$ for each $t>0$.  We say {\bf $A^N_t$ converges to $a_t$ in finite-dimensional
noncommutative distributions} if, for each
$n\in\mathbb{N}$ and times $t_1,\ldots,t_n\ge0$, and each noncommutative polynomial $P$ in $2n$ indeterminates, then almost surely
\[ \lim_{N\to\infty}\textstyle{\frac1N}\mathrm{Trace}\!\left[P(A^N_{t_1},\ldots,A^N_{t_n},(A^N_{t_1})^\ast,\ldots,(A^N_{t_n})^\ast)\right]
= \tau\!\left[P(a_{t_1},\ldots,a_{t_n},a_{t_1}^\ast,\ldots,a_{t_n}^\ast)\right]. \]
\end{definition}
The limit of $X^N_t$ was identified by Voiculescu \cite{Voiculescu}; it is known as {\bf free additive Brownian motion} $x_t$,
and can be constructed on a Fock space.  From here, one can derive the $C^N_t$ case by noting that
$C^N_t = \frac{1}{\sqrt{2}}(X^N_t + iY^N_t)$ where $Y^N_t$ is an independent copy
of $X^N_t$.  Given the independence and rotational invariance, it follows by standard results on asymptotic freeness 
that the large-$N$ limit of $C^N_t$ can be represented as $c_t = \frac{1}{\sqrt{2}}(x_t + iy_t)$ where $\{x_t,y_t\}$ are
{\em freely} independent free additive Brownian motions; we call $c_t$ a {\bf free circular Brownian motion}.

Since the unitary and general linear Brownian motions $U^N_t$ and $B^N_t$ are defined as solutions of SDEs involving
$X^N_t$ and $C^N_t$, good candidates for their large-$N$ limits are given by {\em free} SDEs involving
$x_t$ and $c_t$. (Free stochastic analysis was introduced in \cite{BS1} and further developed in \cite{BS2,KNPS}; the reader
may also consult the background sections of \cite{CK,KempLargeN} for succinct summaries of the relevant concepts.)  The {\bf free unitary Brownian motion} $u_t$ and {\bf free multiplicative Brownian motion} $b_t$ are defined
as solutions to the free SDEs
\begin{equation} \label{e.free.SDEs} du_t = iu_t\,dx_t - \textstyle{\frac12}u_t\,dt \qquad \text{and} \qquad db_t = b_t\,dc_t \end{equation}
(both started at $1$), mirroring the matrix SDEs of \eqref{e.UB.SDEs}.

The free unitary Brownian motion $u_t$ was introduced by Biane in \cite{BianeFields}, wherein he also showed that
it is the large-$N$ limit of the unitary Brownian motion $U^N_t$ as a process (as in Definition \ref{def.limit.process}).
In particular, in the case of a single time $t$ ($n=1$ in the definition, $t_1=t$), since $u_t^\ast = u_t^{-1}$, the statement
is simply that the trace moments $\frac1N\mathrm{Trace}[(U_t^N)^k]$, $k\in\mathbb{Z}$, converge almost surely
as $N\to\infty$ to $\tau[u_t^k]$. The numbers $\nu_k(t):=\tau[u_t^k]$, meanwhile, are the moments of a probability measure $\nu_t$ on the unit circle.
Biane computed these limit moments, which had already appeared in work of Singer \cite{Singer} in Yang--Mills theory
on the plane in an asymptotic regime.  They are given by
\begin{equation}
\nu_{k}(t):=\int_{\partial\mathbb{D}}\omega^{k}~\nu_{t}(d\omega)=e^{-\frac{\left\vert
k\right\vert }{2}t}\sum_{j=0}^{\left\vert k\right\vert -1}\frac{(-t)^{j}}%
{j!}\left\vert k\right\vert ^{j-1}\binom{\left\vert k\right\vert }{j}\label{nutMoments}%
\end{equation}
for $n\in\mathbb{Z}\setminus\{0\}$, with $\nu_0(t) =1$.

From here, using complex analytic techniques, Biane
completely determined the measure $\nu_t$.  For $t>4$, the support of $\nu_t$ is the whole unit circle while for $t<4$ its support is the following arc:
\begin{equation} \label{e.supp.nut} \mathrm{supp}\,\nu_t = \left\{e^{i\theta}\colon |\theta| \le \frac12\sqrt{t(4-t)}+\arccos\left(1-\frac{t}{2}\right)\right\},\quad t<4. \end{equation}
Biane also gave an implicit description of the measure $\nu_t$, which has a real analytic density on the interior of its
support, but we do not need this description presently.  

The key to analyzing $\nu_t$
was determining a certain analytic transform of $\nu_t$ in the unit disk $\mathbb{D}$.  Let
\[ \psi_t(z) := \int_{\partial\mathbb{D}} \frac{\omega z}{1-\omega z}\,\nu_t(d\omega), \qquad z\in\mathbb{D} \]
denote the moment generating function (with no constant term).  The function $\psi_t$ has a continuous extension to
the closed disk $\overline{\mathbb{D}}$; this is tantamount to the fact, as Biane proved, that $\nu_t$ possesses a
continuous density on $\partial\mathbb{D}$. Of greater computational use is the following function:
\begin{equation} \label{e.chit} \chi_t(z) := \frac{\psi_t(z)}{1+\psi_t(z)} \end{equation}
which also has a continuous extension to $\overline{\mathbb{D}}$.  In fact, $\chi_t$ is one-to-one on the open disk, and
its inverse is analytic, with the following simple explicit formula:
\begin{equation} \label{e.ft} f_t(z) := \chi_t^{\langle -1\rangle}(z) = ze^{\frac{t}{2}\frac{1+z}{1-z}}. \end{equation}
It is from this identity that the explicit formulas \eqref{nutMoments} and \eqref{e.supp.nut} are derived.

Biane also introduced the free multiplicative Brownian motion process $b_t$ in \cite{BianeJFA}, where he conjectured
that it is the large-$N$ limit of the $\mathrm{GL}(N;\mathbb{C})$ Brownian motion $B^N_t$.  Given the non-normality of
the matrices and operators involved, this turned out to be a difficult problem that took nearly two decades to solve;
the second author proved this in \cite{KempLargeN}.  Now, the ``holomorphic'' moments $\tau[b_t^k]$ are easily shown to
have the value $1$ for all $k$ (see also \eqref{bstMoments}). But since the process $b_t$ is not normal, these moments do not determine much
of the noncommutative distribution.  The free SDE \eqref{e.free.SDEs} that defines $b_t$ allows for any mixed
moment in $b_t$ and $b_t^\ast$ to be computed (iteratively) given enough patience; see \cite[Proposition 1.8]{KempLargeN}
for some notable examples.  There is, at present, no known simple description of the full noncommutative
distribution of this complicated object.  Its spectrum is also unknown.

\subsection{The domains $\Sigma_{t}$\label{domains.sec}}%

\begin{figure}[htpb]%
\centering
\includegraphics[scale=0.6]
{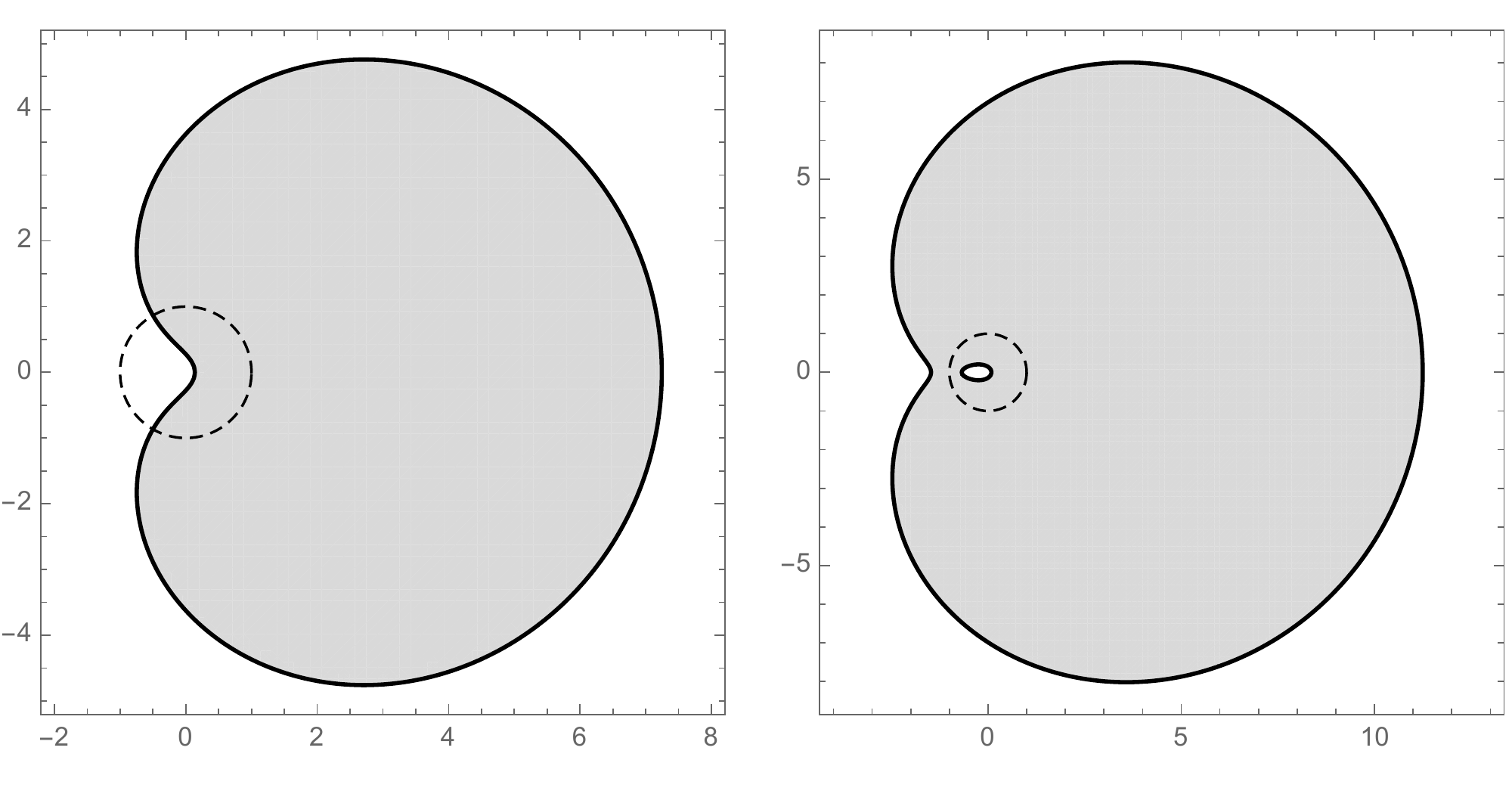}%
\caption{The domains $\Sigma_{t}$ with $t=3$ and $t=4.05$, with the unit
circle (dashed) shown for comparison}%
\label{t3withcircle.fig}%
\end{figure}

In this section, we describe the family of domains $\Sigma_{t}\subset
\mathbb{C}$, introduced by Biane in \cite{BianeJFA}, which enter into the
statement of our main result. They arose in the context of the free Segal--Bargmann
transform (see Section \ref{BianesTransform.sec}), which connects $u_t$ and $b_t$.
For this reason, they are related to the function $f_t$ in \eqref{e.ft}, which is the inverse of the shifted moment generating function of the spectral measure $\nu_t$ of the free unitary Brownian motion.

It is easily verified that if $\left\vert z\right\vert =1$, then $\left\vert
f_{t}(z)\right\vert =1$. There are, however, points $z$ with $\left\vert
z\right\vert\ne 1$ for which $\left\vert f_{t}(z)\right\vert $
is nevertheless equal to 1.

\begin{proposition}
For all $t>0$, consider the set%
\[
E_{t}=\overline{\left\{  z\in\mathbb{C}\left\vert \left\vert z\right\vert
\neq1,~\left\vert f_{t}(z)\right\vert =1\right.  \right\}  }%
\]
and define $\Sigma_t$ to be the connected component of the complement of $E_t$ containing 1. Then $\Sigma_t$ is bounded for all $t>0$, $\Sigma_t$ is simply connected for $t\leq 4$, and $\Sigma_t$ is doubly connected for $t>4$. In all cases, we have 
\[
\partial\Sigma_{t}=E_{t}.
\]

\end{proposition}

%

These properties of the region $\Sigma_t$ were proved  by Biane in \cite{BianeJFA}; see especially pp. 273--274. See also \cite[Section 4.2]{Ho}.
The closure in the definition of the set $E_{t}$ is needed to fill in the points (at most two of them)
where $\partial\Sigma_{t}$ intersects the unit circle.  In recent joint work between the present authors and Driver \cite[Theorem 4.1]{DHK-Brown-Measure}, we found a simpler description of the regions $\Sigma_t$. They are the level sets of a certain explicit function: $\Sigma_t = \{\lambda\in\mathbb{C}\colon T(\lambda)<t\}$ and $\partial\Sigma_t = \{\lambda\in\mathbb{C}\colon T(\lambda)=t\}$, where
\[ T(\lambda) = |\lambda-1|^2\frac{\log(|\lambda|^2)}{|\lambda|^2-1}. \]
(It is easy to compute that this expression has a limit as $|\lambda|\to 1$. Indeed, $T$ extends to a real analytic function on $\mathbb{C}\setminus\{0\}$, and for $|\lambda|=1$, we have $T(\lambda) = |\lambda-1|^2$.)  Figure \ref{t3withcircle.fig} shows the domain
$\Sigma_{t}$ with $t=3$ and $t=4.05$, with the unit circle shown for comparison. Figure \ref{t4region.fig}
then shows the transitional case $t=4$ in more detail. In all cases, 1 is in $\Sigma_t$ and 0 is not in $\overline{\Sigma}_t$.

\begin{figure}[htpb]%
\centering
\includegraphics[scale=0.46]{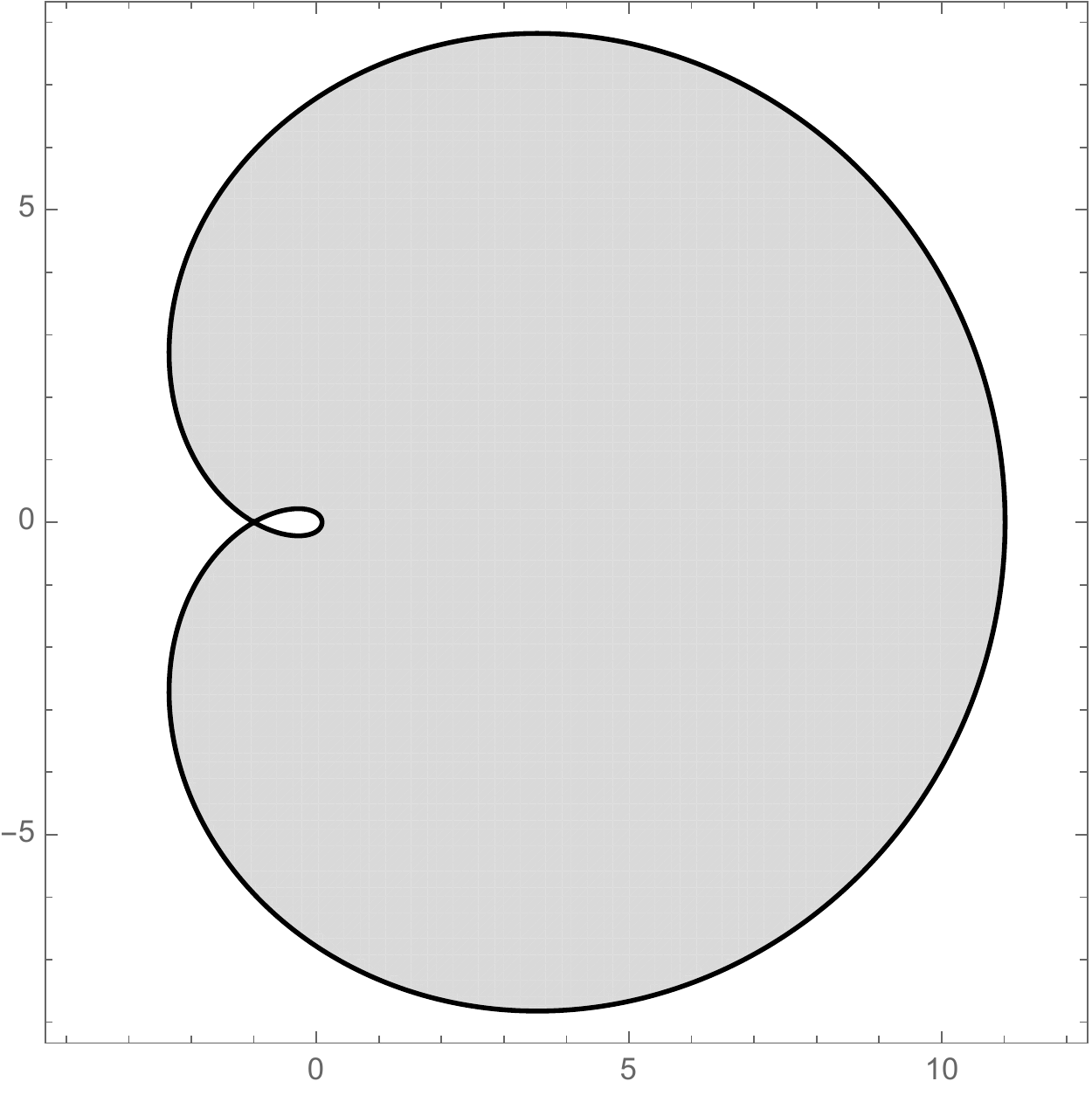}%
\quad\,
\includegraphics[scale=0.477]{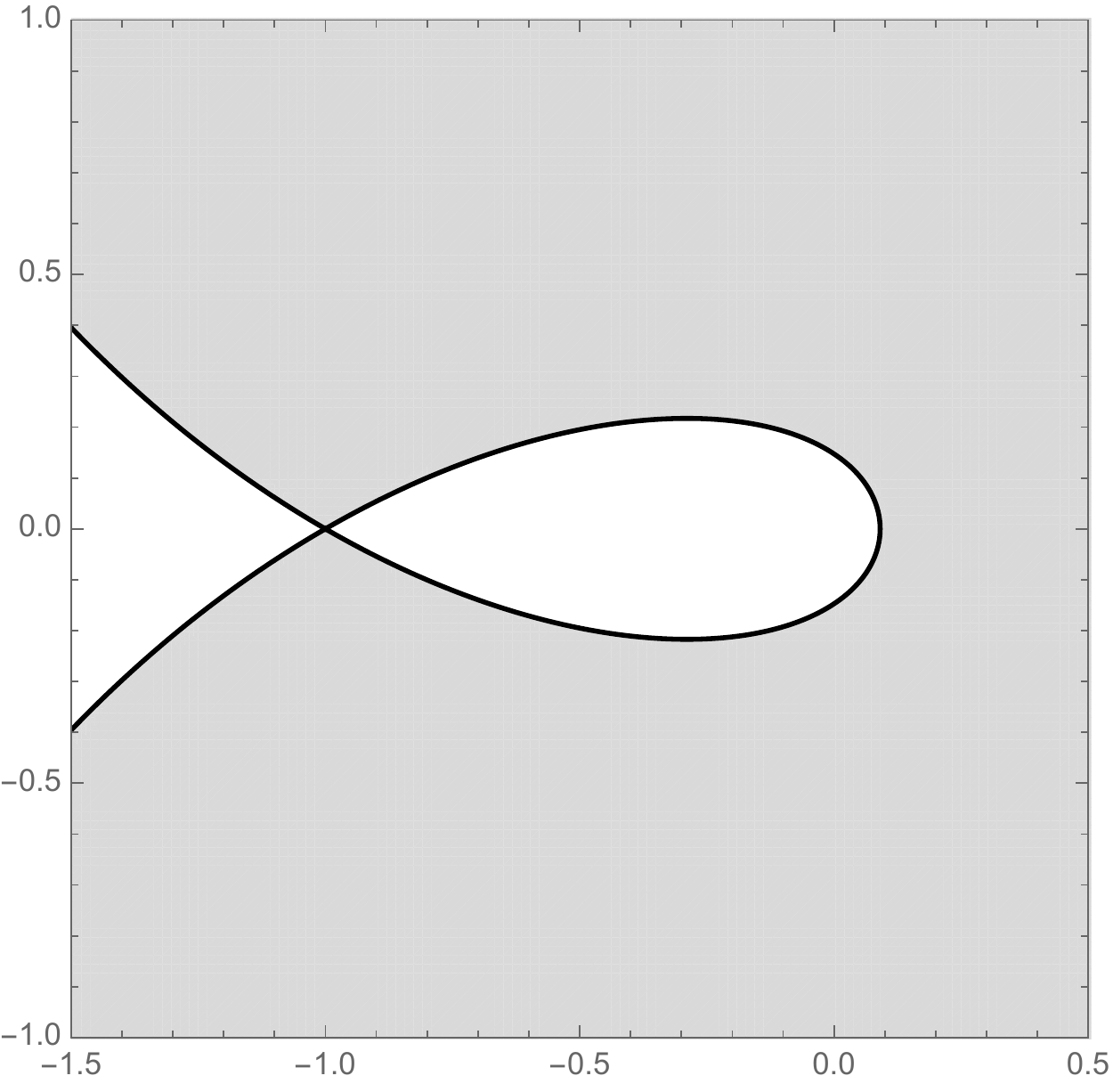}%
\caption{The region $\Sigma_{t}$ with $t=4$ (left) and a detail thereof
(right)}%
\label{t4region.fig}%
\end{figure}

An important property of the region $\Sigma_t$, which follows from the just-cited results of Biane \cite{BianeJFA}, involves
the support arc of the spectral measure $\nu_t$ of free unitary Brownian motion.  This
result is crucial to the proof of our main theorem.

\begin{proposition}
\label{ftMaps.prop}For all $t>0$, the function $f_{t}$ maps $\mathbb{C}\setminus\overline{\Sigma}_{t}$
injectively onto $\mathbb{C}\setminus\mathrm{supp}\,\nu_{t}$.
\end{proposition}

This is a typical ``slit plane'' conformal map; see Figure \ref{f3maps.fig}.%

\begin{figure}[htpb]%
\centering
\includegraphics[scale=0.4]
{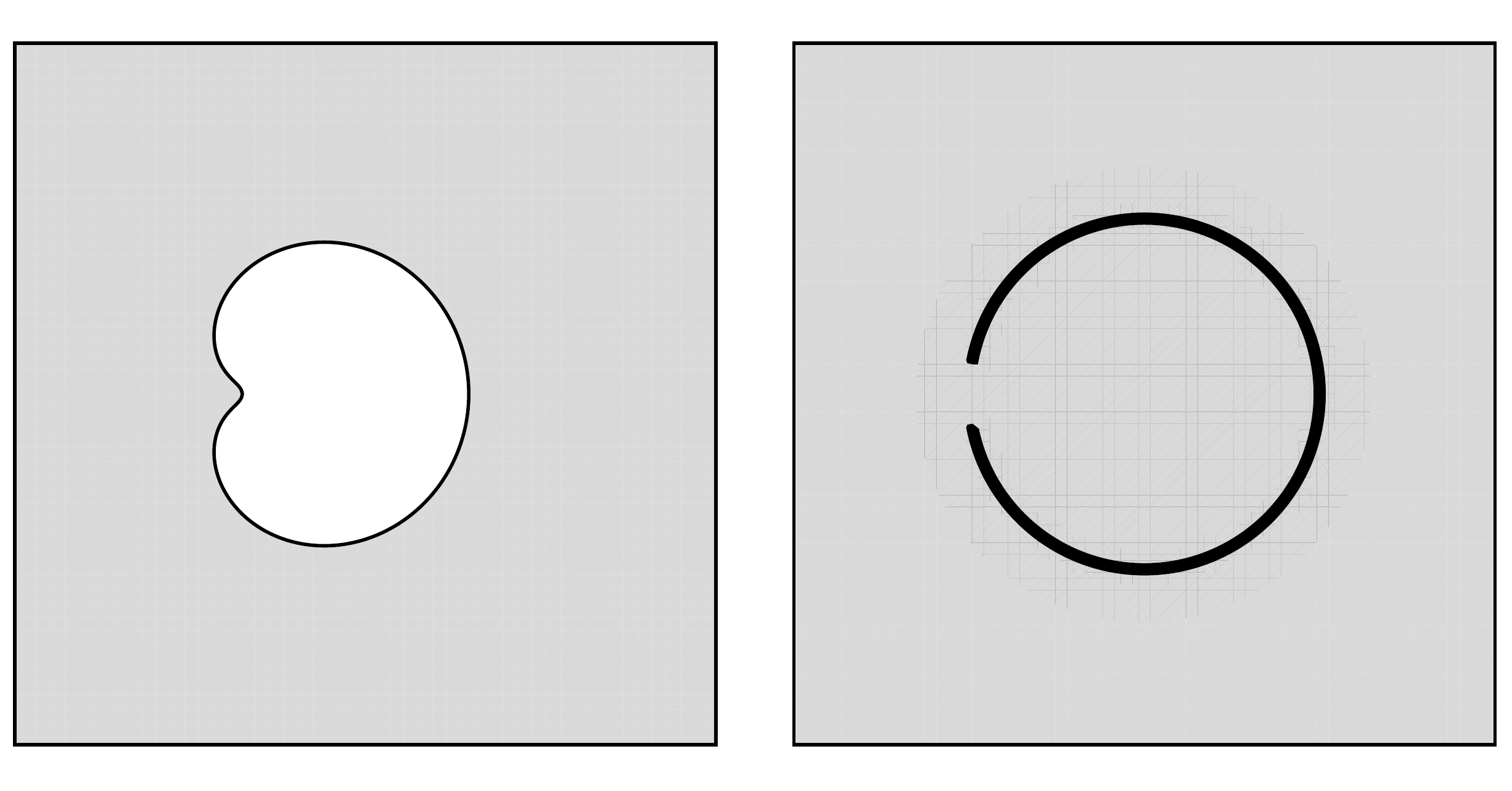}%
\caption{The map $f_{t}$ takes $\mathbb{C}\setminus\overline{\Sigma}_{t}$ (left)
injectively onto $\mathbb{C}\setminus\mathrm{supp}\,\nu_{t}$ (right). Shown
for $t=3$}%
\label{f3maps.fig}%
\end{figure}

\subsection{Brown measure\label{section Brown measure}}

We work in the context of a sufficiently rich noncommutative probability space: a tracial von Neumann algebra.

\begin{definition}
\label{noncom.def}
A \textbf{tracial von Neumann algebra} is a finite von Neumann algebra $\mathcal{A}$ together with a faithful, normal, tracial state $\tau:\mathcal{A}\rightarrow\mathbb{C}$.
\end{definition}

Recall that a \textit{state} $\tau$ is norm-$1$ linear functional taking non-negative elements to non-negative real numbers. (Such a functional necessarily satisfies $\tau(1)=\|\tau\|=1$.) A state $\tau$ is called \textit{faithful} if $\tau(a^*a)>0$ for all $a\neq 0$, it is called \textit{normal} if it is continuous with respect to the weak$^\ast$ topology on $\mathcal{A}$ (cf.\ \cite[Theorem III. 2.1.4, p.\ 262]{Blackadar}), and it is called \textit{tracial} if $\tau(ab)=\tau(ba)$ for all $a,b\in\mathcal{A}$.

Let $(\mathcal{A},\tau)$ be a tracial von Neumann algebra. For each element $a$ of
$\mathcal{A}$, it is possible to define a probability measure $\mu_{a}$ on
$\mathbb{C}$ called the \textbf{Brown measure} of $a$, which should be
interpreted as something like an empirical eigenvalue distribution for the
operator $a$.  The definitions and properties stated in this section may be
found in Brown's original paper \cite{Br} and in \cite[Chapter 11]{MS}.

We first recall the notion of the \textbf{Fuglede--Kadison determinant} of $a$
\cite{FK1,FK2}, denoted $\Delta(a)$, which is most easily defined by a
limiting process:%
\begin{equation}
\log\Delta(a)=\lim_{\varepsilon\rightarrow0}\frac{1}{2}\tau\lbrack\log
(a^{\ast}a+\varepsilon)]. \label{FKdet}%
\end{equation}
In general, $\log\Delta(a)$ may have the value $-\infty$, in which case,
$\Delta(a)=0$. If, for example, $\mathcal{A}=M_N(\mathbb{C})$ and $\tau$
is the normalized trace, then 
$\Delta(a)=\left\vert\det a\right\vert ^{1/N}$, where $\det a$ is the ordinary determinant of $a$.

For a tracial von Neumann algebra $(\mathcal{A},\tau)$, the Brown measure of an element
$a\in\mathcal{A}$ is then defined as%
\[
\mu_{a}=\frac{1}{2\pi}\nabla_{\lambda}^{2}\log(\Delta(a-\lambda)),
\]
where $\nabla_{\lambda}^{2}$ is the Laplacian with respect to $\lambda$,
computed in the distributional sense. It can be shown that this distributional
Laplacian is a represented by a probability measure on the plane.

\begin{proposition}
\label{brownSupport.prop}Let $a$ be an element of $\mathcal{A}$ and let
$\mu_{a}$ be its Brown measure. Then the following results hold.

\begin{enumerate}
\item \label{BrownPropertyProb}The measure $\mu_{a}$ is a probability measure
on the plane.

\item \label{BrownPropertySupport}The support of $\mu_{a}$ is contained
in the spectrum of $a$, but the two sets do not coincide in general.

\item \label{BrownPropertyMoments}For all non-negative integers $n$,
\[
\int z^{n}~d\mu_{a}(z)=\tau\lbrack a^{n}],
\]
and if $a$ is invertible, the same result holds for all integers $n$.
\end{enumerate}
\end{proposition}

Although the support of the Brown measure can
be a proper subset of the spectrum, there are many interesting examples in
which the two sets coincide.

Using the limiting formula \eqref{FKdet} for the Flugede--Kadison determinant,
we may give a limiting formula for the Brown measure. With the notation%
\[
a_{\lambda}:=a-\lambda,
\]
we have%
\begin{equation}
\mu_{a}=\frac{1}{4\pi}\lim_{\varepsilon\rightarrow0}\left\{  \nabla_{\lambda
}^{2}\tau\lbrack\log(a_{\lambda}^{\ast}a_{\lambda}+\varepsilon)]~d^{2}%
\lambda\right\}  , \label{limitFormula1}%
\end{equation}
where $d^{2}\lambda$ is the two-dimensional Lebesgue measure on $\mathbb{C}$
and the limit is in the weak sense. Furthermore, the Laplacian on the
right-hand side of \eqref{limitFormula1} can be computed explicitly
\cite[Section 11.5]{MS}, giving still another formula for the Brown measure:%
\begin{equation}
\mu_{a}=\frac{1}{\pi}\lim_{\varepsilon\rightarrow0}\left\{  \varepsilon
~\tau\lbrack(a_{\lambda}^{\ast}a_{\lambda}+\varepsilon)^{-1}(a_{\lambda
}a_{\lambda}^{\ast}+\varepsilon)^{-1}]~d^{2}\lambda\right\}  .
\label{limitFormula2}%
\end{equation}

The following result follows easily from \eqref{limitFormula2}.

\begin{corollary}
\label{BrownSupport.cor}Suppose the quantity
\begin{equation}
\tau\lbrack(a_{\lambda}^{\ast}a_{\lambda}+\varepsilon)^{-1}(a_{\lambda
}a_{\lambda}^{\ast}+\varepsilon)^{-1}] \label{boundedQuantity}%
\end{equation}
is bounded uniformly for all $\varepsilon>0$ and all $\lambda$ in a
neighborhood of some value $\lambda_{0}$. Then $\lambda_{0}$ does not belong
to the support of the Brown measure $\mu_{a}$.
\end{corollary}

In particular, if $\lambda_{0}$ belongs to the resolvent set of $a$, it is not
hard to see that the quantity \eqref{boundedQuantity} has a finite limit as
$\varepsilon\to 0$, for all $\lambda$ in a neighborhood of
$\lambda_{0}$, so that the corollary applies. Thus, Corollary
\ref{BrownSupport.cor} implies Point \ref{BrownPropertySupport} of Proposition
\ref{brownSupport.prop}. But the corollary is stronger, in the sense that it may apply even if $\lambda_0$ is in the spectrum of $a$.

We close this section by noting three important special cases of the Brown measure.

\begin{itemize}
\item When $\mathcal{A}=M_N(\mathbb{C})$ and $\tau$
is the normalized trace, the Brown measure of a matrix $A$ is its empirical
eigenvalue distribution. That is,%
\[
\mu_{A}=\frac{1}{N}\sum_{j=1}^{n}\delta_{\lambda_{j}},
\]
where $\lambda_{1},\ldots,\lambda_{N}$ are the eigenvalues of $A$, listed with
their algebraic multiplicity.

\item For a normal element $a$ of $\mathcal{A}$, Brown measure coincided with the
spectral measure: $\mu_a$ is just the composition of the projection-valued spectral resolution with $\tau$:%
\[
\mu_{a}(V)=\tau(E^{a}(V)), \qquad V\in\mathrm{Borel}(\mathbb{C}),
\]
where $E^{a}$ is the projection-valued measure associated to $a$ by the
spectral theorem.

\item {\em $\mathscr{R}$-diagonal operators} form an important class of (generally) non-normal elements of a tracial von Neumann algebra; these include the circular operators $c_t$ and Haar unitaries.  An element
$a\in\mathcal{A}$ is $\mathscr{R}$-diagonal if it has the same non-commutative distribution as $ua$
for any Haar unitary operator $u$ freely independent from $a$; see \cite[Lecture 15]{NS} for more details.
In \cite{HaagerupLarsen}, Haagerup and Larsen proved that the Brown measure of an
$\mathscr{R}$-diagonal operator is rotationally invariant, with a radial real analytic density supported on a certain annulus (or circle) determined by $a$; see the discussion following Theorem \ref{L22support.thm} below for more details.

\end{itemize}

Since the free multiplicative Brownian motion $b_{t}$ is not finite-dimensional, normal, or $\mathscr{R}$-diagonal, none of the preceding cases applies.  We will see, however, that some of the ideas related to the support of the Brown measure of $\mathscr{R}$-diagonal operators are useful in the present context.

\section{Free Segal--Bargmann transform\label{BianesTransform.sec}}

Recall from the Section \ref{section SDEs} that the law $\nu_{t}$ of free unitary
Brownian motion is a probability
measure on the unit circle that represents the limiting empirical eigenvalue
distribution for Brownian motion in the unitary group. In \cite{BianeJFA},
Biane introduced a \textquotedblleft free Hall transform\textquotedblright\ $\mathscr{G}_{t}$
that maps $L^{2}(\partial\mathbb{D},\nu_{t})$ into $\mathcal{H}%
(\Sigma_{t})$, the space of holomorphic functions on the domain $\Sigma_{t}$.
In this section, we recall both the original construction of $\mathscr{G}_{t}$
given by Biane and a realization given by the authors together with Driver
\cite{DHKLargeN} and C\'ebron \cite{Ceb}. The transform $\mathscr{G}_{t}$
will be a crucial tool in the proof of our main theorem.

\subsection{Using free probability\label{freeSBT.sec}}

Let $u_{t}$ be a free unitary Brownian motion and $b_{t}$ a free
multiplicative Brownian that is freely independent from $u_{t}$, both living in
a tracial von Neumann algebra $(\mathcal{B},\tau)$. In the approach pioneered by
Biane and further developed by C\'ebron, the map $\mathscr{G}_{t}$ is characterized by the
requirement that for each Laurent polynomial $p$, we have%
\begin{equation}
(\mathscr{G}_{t}p)(b_{t})=\tau[p(b_{t}u_{t})|b_{t}], \label{freeExpectation}%
\end{equation}
where $\tau[\,\cdot\,|b_{t}]$ is the conditional trace with respect to the algebra
generated by $b_{t}$. (Compare \cite[Theorem 8]{BianeJFA} for a strictly unitary analog, and \cite[Theorem 3]{Ceb} for the precise statement of \eqref{freeExpectation}.) If, for example, $p$ is the polynomial $p(u)=u^{2}$, then it
is not hard to compute (cf.\ \cite[p.\ 55]{MS}) that
\[
\tau[b_{t}u_{t}b_{t}u_{t}|b_{t}]   =\tau(u_{t})^{2}b_{t}^{2}+(\tau(u_{t}%
^{2})-\tau(u_{t})^{2})\tau(b_{t})b_{t}.
\]
We may then use the moment formulas $\tau(u_{t})=e^{-t}$,
$\tau(u_{t}^{2})=e^{-t}(1-t)$, and $\tau(b_{t})=1$. (The moments of $u_{t}$
are the constants $\nu_k(t)$ of \eqref{nutMoments}, while the
moments of $b_{t}$ are the $s=t$ case of \eqref{bstMoments}.) We therefore obtain
\[
\tau[b_{t}u_{t}b_{t}u_{t}|b_{t}]   = e^{-t}(b^2-tb).
\]
Thus, in this
case, $\mathscr{G}_{t}p$ is also a polynomial, given by%
\begin{equation}
(\mathscr{G}_{t}p)(b)=e^{-t}(b^{2}-tb). \label{GtU2}%
\end{equation}

As explained in Section \ref{GtBt}, the map $\mathscr{G}_{t}$ can be viewed as
the large-$N$ limit of the generalized Segal--Bargmann transform over $\mathrm{U}(N)$ introduced by
the first author in \cite{Ha1994}. The motivation for Biane's definition of
$\mathscr{G}_{t}$ is the stochastic approach to the generalized
Segal--Bargmann transform developed by Gross and Malliavin \cite{GM}.

\subsection{As an integral operator}

Using the subordination method developed
in \cite{BianeSubordination}, Biane realized $\mathscr{G}_{t}$ as an
integral operator mapping $L^{2}(\partial\mathbb{D},\nu_{t})$ into $\mathcal{H}(\Sigma
_{t})$. Explicitly,
\begin{equation}
(\mathscr{G}_{t}f)(z)=\int_{\partial\mathbb{D}}f(\omega)\frac{\left\vert 1-\chi_{t}%
(\omega)\right\vert ^{2}}{(z-\chi_{t}(\omega))(z^{-1}-\overline{\chi_{t}(\omega)})}~d\nu
_{t}(\omega),\quad z\in\Sigma_{t}, \label{GtIntegral}%
\end{equation}
where $\chi_{t}$ was defined in \eqref{e.chit}. (See 
\cite[Theorem 8]{BianeJFA} and the computations that follow it, along with Proposition
13.) Here, for $\omega\in \partial\mathbb{D}$, $\chi_{t}(\omega)$ denotes the value of
the unique continuous extension of $\chi_t$ to $\overline{\mathbb{D}}$; in other words,
it is the limiting value of $\chi_{t}(\zeta)$ as $\zeta\in\mathbb{D}$ approaches $\omega$
from \textit{inside} the unit disk. Note
that for $\omega\in\partial\mathbb{D}$ both $\chi_{t}(\omega)$ and $1/\overline{\chi_{t}(\omega)}$ lie
on the boundary of $\Sigma_{t}$, so that the integrand in \eqref{GtIntegral}
is a holomorphic function of $z$ for $z$ in the open set $\Sigma_{t}$.
Biane showed that, for $t\neq4$, the map $\mathscr{G}_{t}$ is injective, so that
it is possible to identify $L^{2}(\partial\mathbb{D},\nu_{t})$ with its image:%
\[
\mathscr{A}_{t}:=\mathrm{Image}(\mathscr{G}_{t})\subset\mathcal{H}(\Sigma
_{t}).
\]

Now, let us define $L_{\mathrm{hol}}^{2}(b_{t},\tau)$ to be the closure in the
noncommutative $L^{2}$ space $L^{2}(\mathcal{B},\tau)$ of the space of
elements of the form $p(b_{t})$, where $p$ is a Laurent polynomial in one
variable. That is to say, $L_{\mathrm{hol}}^{2}(b_{t},\tau)$ is the closure of
the span of the positive and negative integer powers of $b_{t}$, \textit{not}
including any powers of $b_{t}^{\ast}$. (Biane used a slightly different
definition that is easily seen to be equivalent to this one.) Biane showed that
for $t\neq4$, there is a bijection between $\mathscr{A}_{t}$ and
$L_{\mathrm{hol}}^{2}(b_{t},\tau)$ uniquely determined by the condition that
for each Laurent polynomial $p$, we have%
\[
p\mapsto p(b_{t}).
\]
Note that for $t\neq4$, the space $L_{\mathrm{hol}}^{2}(b_{t},\tau)$ is
identified with the space $\mathscr{A}_{t}$ of holomorphic functions.
Nevertheless, the noncommutative $L^{2}$ norm on $L_{\mathrm{hol}}^{2}%
(b_{t},\tau)$ does not correspond to an $L^{2}$ norm on $\mathscr{A}_{t}$ with
respect to any measure on $\Sigma_{t}$.  (It is, instead, the Hilbert space norm
induced by a certain reproducing kernel on $\mathscr{A}_t$ which is defined
by the integral kernel of $\mathscr{G}^t$, cf.\ \eqref{GtIntegral}.)

For a general $f\in\mathscr{A}_{t}$, we will write the corresponding element
of $L_{\mathrm{hol}}^{2}(b_{t},\tau)$ suggestively as $f(b_{t})$ and think of
the map $f\mapsto f(b_{t})$ as a variant of the usual holomorphic functional
calculus. That is to say, we think of the map from $\mathscr{A}_{t}$ to
$L_{\mathrm{hol}}^{2}(b_{t},\tau)$ as \textquotedblleft evaluation on $b_{t}%
$.\textquotedblright\ Note, however, that elements of $L_{\mathrm{hol}}%
^{2}(b_{t},\tau)$ are in general unbounded operators.

\begin{theorem}
[Biane's Free Hall Transform]\label{bianeSBT.thm}For all $t>0$ with
$t\neq4$, the map $\mathscr{G}_{t}$ is a unitary isomorphism from
$L^{2}(\partial\mathbb{D},\nu_{t})$ to $\mathscr{A}_{t}$, where the norm on
$\mathscr{A}_{t}$ is defined by identification with $L_{\mathrm{hol}}^{2}(b_{t},\tau)$.
In particular, we have%
\[
\left\Vert f\right\Vert _{L^{2}(\partial\mathbb{D},\nu_{t})}=
\tau\{[(\mathscr{G}_{t}f)(b_t)]^* (\mathscr{G}_{t}f)(b_t)\}
\]
for all $f\in L^{2}(\partial\mathbb{D},\nu_{t})$.
\end{theorem}

When $t=4$, the preceding theorem is not known to hold, because it is not
known that $\mathscr{G}_{t}$ is injective. But one still has a theorem, as
follows. One considers at first the map $p\mapsto\mathscr{G}_{t}p$ on
polynomials and then constructs a map from the space of polynomials into
$L_{\mathrm{hol}}^{2}(b_{t},\tau)$ by mapping $p$ to $(\mathscr{G}_{t}%
p)(b_{t})$. This map is isometric for all $t>0$ and it extends to a unitary
map of $L^{2}(\partial\mathbb{D},\nu_{t})$ onto $L_{\mathrm{hol}}^{2}(b_{t},\tau)$; see
Section \ref{computeL2spec.sec} for details.

In light of the preceding discussion, we expect that, at least for $t\neq4$,
the spectrum of $b_{t}$ will \textit{not} be contained in $\Sigma_{t}$. After
all, if such a containment held, the operator $f(b_{t})$, $f\in\mathscr{A}%
_{t}\subset\mathcal{H}(\Sigma_{t})$ would presumably be computable by the
holomorphic functional calculus, in which case $f(b_{t})$ would be a bounded
operator. But actually, \textit{every} element of $L_{\mathrm{hol}}^{2}%
(b_{t},\tau)$ arises as $f(b_{t})$ for some $f\in\mathcal{A}_{t}$, and the
elements of $L_{\mathrm{hol}}^{2}(b_{t},\tau)$ are in general unbounded operators.

On the other hand, since we are able to define $f(b_{t})$ for any $f\in\mathscr{A}_{t}$,
at least as an unbounded operator, it seems reasonable to
expect that the spectrum of $\Sigma_{t}$ is contained in the \textit{closure
of }$\Sigma_{t}$. Our main result, that the Brown measure of $b_{t}$ is
supported in $\overline{\Sigma}_{t}$, is a step toward establishing this claim;
compare Proposition \ref{brownSupport.prop}.

\subsection{From the generalized Segal--Bargmann transform\label{GtBt}}

In 1994, the first author introduced a generalized Segal--Bargmann transform
for compact Lie groups \cite{Ha1994}. In the case of the unitary group $\mathrm{U}(N)$,
the transform, which we denote here as $\mathscr{G}_{t}^{N}$, maps $L^{2}(\mathrm{U}(N),\rho_{t})$ to the
space of \textit{holomorphic} functions in $L^{2}(\mathrm{GL}(N;\mathbb{C}),\gamma
_{t})$.  (Note: in \cite{Ha1994} and follow-up work such as \cite{DHKLargeN},
the transform was often denoted $B_t$; to avoid clashing with our present notation
$B^N_t$ for the Brownian motion on $\mathrm{GL}(N;\mathbb{C})$, we use $\mathscr{G}_t^N$ instead
for the Segal--Bargmann transform here.) Here $\rho_{t}$ and $\gamma_{t}$ are \textit{heat kernel
measures}---that is, the distributions at time $t$ of Brownian motions on the
respective groups, starting at the identity. The transform is defined as%
\begin{equation}
\mathscr{G}_{t}^{N}f=(e^{t\Delta/2}f)_{\mathbb{C}}, \label{btnDef}%
\end{equation}
where $\Delta$ is the Laplacian on $\mathrm{U}(N)$, $e^{t\Delta/2}$ is the associated
(forward) heat operator, and $(\cdot)_{\mathbb{C}}$ denotes the holomorphic
extension of a sufficiently nice function from $\mathrm{U}(N)$ to $\mathrm{GL}(N;\mathbb{C})$.
See also \cite{Hall2001} for more information. The transform can easily be
\textquotedblleft boosted\textquotedblright\ to map matrix-valued functions
on~$\mathrm{U}(N)$ to holomorphic matrix-valued functions on $\mathrm{GL}(N;\mathbb{C})$
(by acting component-wise; i.e.,\ via $\mathscr{G}^N_t\otimes\mathbf{1}_{M_N(\mathbb{C})}$).

A stochastic approach to the transform was developed by Gross and Malliavin in
\cite{GM}; this approach played an important role in Biane's paper
\cite{BianeJFA}. See also \cite{Halll2003,HS} for further
development of the ideas in \cite{GM}.  Let $U_{t}^{N}$ and $B_{t}^{N}$ be
independent Brownian motions in $\mathrm{U}(N)$ and $\mathrm{GL}(N;\mathbb{C})$ (cf.\ \eqref{e.UB.SDEs}),
and let $f$ be a function on $\mathrm{U}(N)$ that admits a holomorphic extension (also
denoted $f$) to $\mathrm{GL}(N;\mathbb{C})$. Then we have
\begin{equation}
\mathbb{E}[f(B_{t}^{N}U_{t}^{N})|B_{t}^{N}]=(\mathscr{G}_{t}^{N}f)(B_{t}^{N}).
\label{conditioningGM}%
\end{equation}
This result, by itself, is not deep. After all, in the finite-dimensional
case, the conditional expectation can be computed as an expectation with
respect to $U^N_{t}$, with $B^N_{t}$ treated as a constant. Since $U^N_{t}$ is
distributed as a heat kernel on $\mathrm{U}(N)$, the left-hand side of
\eqref{conditioningGM} becomes a convolution of $f$ with the heat kernel,
giving the heat kernel in the definition \eqref{btnDef} of the transform
$\mathscr{G}_{t}^{N}$.

The crucial next step in \cite{GM} is to regard $U^N_{t}$ and $B^N_{t}$ as
functionals of Brownian motions in the Lie algebra, by solving the relevant versions of the
stochastic differential equation \eqref{e.BM.Lie.Group}. Using this idea, Gross and
Malliavin are able to deduce the properties of the \textit{generalized}
Segal--Bargmann from the previously known properties of the \textit{classical}
Segal--Bargmann transform for an infinite-dimensional linear space, namely the
path space in the Lie algebra of $\mathrm{U}(N)$. (We are glossing over certain
technical distinctions; the preceding description is actually closer to
\cite[Theorem 18]{HS}.) The expression \eqref{conditioningGM} was the
motivation for Biane's formula \eqref{freeExpectation} in the free case, and
just as in \cite{GM}, Biane was able to obtain properties of the transform
$\mathscr{G}_{t}$ from the corresponding linear case.

In \cite{BianeJFA}, Biane conjectured, with an outline of a proof, that the
free Hall transform $\mathscr{G}_{t}$ can be realized using the large-$N$
limit of $\mathscr{G}_{t}^{N}$. This conjecture was then verified independently by the
authors and Driver in \cite{DHKLargeN} and by C\'{e}bron in \cite{Ceb}; see
also the expository paper \cite{HallExpository}.

The limiting process is as follows. Consider the transform $\mathscr{G}_{t}^{N}$ on
matrix-valued functions of the form $f(U)$, where $f$ is a function on the
unit circle and $f(U)$ is computed by the functional calculus. If, for
example, $f$ is the function $f(u)=u^{2}$ on the circle, then we can consider
the associated matrix-valued function $f(U)=U^{2}$ on the unitary group
$\mathrm{U}(N)$. For any fixed $N$, the transformed function $\mathscr{G}_{t}^{N}(f)$ on
$\mathrm{GL}(N;\mathbb{C})$ will no longer be of functional-calculus type.
Nevertheless, \textit{in the large-}$N$ \textit{limit}, $\mathscr{G}_{t}^{N}$ will map
$f(U)$ to the functional-calculus function $(\mathscr{G}_{t}f)(Z)$, $Z\in
\mathrm{GL}(N;\mathbb{C})$.

Specifically, if $p$ is a Laurent polynomial, then $\mathscr{G}_{t}p$ is also
a Laurent polynomial, and (abusing notation slightly)
\[
\mathscr{G}_{t}^{N}(p(U))=(\mathscr{G}_{t}p)(Z)+O(1/N^{2}),\quad Z\in \mathrm{GL}(N;\mathbb{C}),
\]
where $O(1/N^{2})$ denotes a term whose norm is bounded by a constant times
$1/N^{2}$. See \cite[Theorem 1.11]{DHKLargeN} and \cite[Theorem 4]{Ceb}.
In particular, if $f(U)=U^{2}$, then in light of \eqref{GtU2}, we have%
\[
(\mathscr{G}_{t}^{N}f)(Z)=e^{-t}(Z^{2}-tZ)+O(1/N^{2}),\quad Z\in \mathrm{GL}(N;\mathbb{C}).
\]
(See also Example 3.5 and the computations on p. 2592 of \cite{DHKLargeN}.)

\section{An outline of the proof of Theorem \ref{main.thm}\label{outline.sec}}

As we pointed out in Proposition \ref{brownSupport.prop}, the Brown measure of
an operator $a$ is supported on the spectrum of $a$. We strengthen this
result, as follows. Given a noncommutative probability space $(\mathcal{A},\tau)$, we can construct the noncommutative $L^{2}$ space
$L^{2}(\mathcal{A},\tau)$, which is the completion of $\mathcal{A}$ with
respect to the noncommutative $L^{2}$ inner product, $\left\langle
a,b\right\rangle =\tau(b^{\ast}a)$. It makes sense to multiply an element of
the noncommutative $L^{2}$ space $L^{2}(\mathcal{A},\tau)$ by an element of
$\mathcal{A}$ itself, and the result is again in $L^{2}(\mathcal{A},\tau)$. We
say that $a\in\mathcal{A}$ \textbf{has an inverse in }$L^{2}$ if there exists
some $b\in L^{2}(\mathcal{A},\tau)$ such that $ab=ba=1$.

\begin{theorem}
\label{BrownAndL2Spec.thm}Let $(\mathcal{A},\tau)$ be a tracial von Neumann algebra and let $\lambda_{0}$ be in $\mathbb{C}$.
Suppose that $(a-\lambda)^{2}$ has an inverse---denoted $(a-\lambda)^{-2}%
$---in $L^{2}(\mathcal{A},\tau)$ for all $\lambda$ in a neighborhood of
$\lambda_{0}$ and that $\left\Vert (a-\lambda)^{-2}\right\Vert _{L^{2}%
(\mathcal{A},\tau)}$ is bounded near $\lambda_{0}$. Then $\lambda_{0}$ does
not belong to the support of the Brown measure $\mu_{a}$.
\end{theorem}

Note that if $a-\lambda_{0}$ has a \textit{bounded} inverse---that is, an
inverse in $\mathcal{A}$---then $a-\lambda$ also has an inverse for all
$\lambda$ in a neighborhood of $\lambda_{0}$, and $\left\Vert (a-\lambda
)^{-1}\right\Vert _{\mathcal{A}}$ is bounded near $\lambda_{0}$. In that case,
we have%
\[
\left\Vert (a-\lambda)^{-2}\right\Vert _{L^{2}(\mathcal{A},\tau)}%
\leq\left\Vert (a-\lambda)^{-2}\right\Vert _{\mathcal{A}}\leq\left\Vert
(a-\lambda)^{-1}\right\Vert _{\mathcal{A}}^{2},
\]
which shows that $\left\Vert (a-\lambda)^{-2}\right\Vert _{L^{2}%
(\mathcal{A},\tau)}$ is bounded near $\lambda_{0}$. Thus, we can recover from
Theorem \ref{BrownAndL2Spec.thm} the result that the support of
$\mu_{a}$ is contained in the spectrum of $a$. In general, however, Theorem
\ref{BrownAndL2Spec.thm} could apply even if $a-\lambda_{0}$ does not have a
bounded inverse.

We now briefly indicate the proof of Theorem \ref{BrownAndL2Spec.thm}. Using
the notation%
\[
a_{\lambda}=a-\lambda,
\]
we make the following intuitive but non-rigorous estimates: for all $\varepsilon>0$,
\begin{align*}
\tau\lbrack(a_{\lambda}^{\ast}a_{\lambda}+\varepsilon)^{-1}(a_{\lambda
}a_{\lambda}^{\ast}+\varepsilon)^{-1}]   \leq\tau\lbrack(a_{\lambda}^{\ast
}a_{\lambda})^{-1}(a_{\lambda}a_{\lambda}^{\ast})^{-1}]
  &=\tau\lbrack a_{\lambda}^{-2}(a_{\lambda}^{-2})^{\ast}] \\
  &=\left\Vert (a-\lambda)^{-2}\right\Vert _{L^{2}(\mathcal{A},\tau)}^{2}.
\end{align*}
(The given estimate actually does hold; the proof is in Section \ref{section L22}.) If the hypotheses of the theorem hold, this last expression is bounded for
$\lambda$ near $\lambda_{0}$. Corollary \ref{BrownSupport.cor} then shows that
$\lambda_{0}$ is not in the support of the Brown measure of $a$.

We now apply Theorem \ref{BrownAndL2Spec.thm} to the case of interest to us,
in which $a$ is taken to be a free multiplicative Brownian motion $b_{t}$ in a tracial von Neumann algebra $(\mathcal{B},\tau)$. Recall that $L_{\mathrm{hol}}^{2}%
(b_{t},\tau)$ denotes the closure in $L^{2}(\mathcal{B},\tau)$ of the space of
Laurent polynomials in the element $b_{t}$.

\begin{theorem}
\label{L2Inverse.thm}For all $t>0$, if $\lambda\in\mathbb{C}\setminus
\overline{\Sigma}_{t}$, then the element $(b_{t}-\lambda)^{n}$ has an inverse in
$L_{\mathrm{hol}}^{2}(b_{t},\tau)\subset L^{2}(\mathcal{B},\tau)$ for all
$n=1,2,3,\ldots$, with local bounds on the $L^{2}$ norm of the inverse.
\end{theorem}

When $t\neq4$, the proof of this lemma draws on the transform $\mathscr{G}_{t}$
in Theorem \ref{bianeSBT.thm}. We will show that the function
$1/(z-\lambda)^{n}$ belongs to the space $\mathscr{A}_{t}$ of holomorphic
functions on $\Sigma_{t}$, at which point Theorem \ref{bianeSBT.thm} tells us
that there is a corresponding element $(b_t-\lambda)^{-n}$, which will be
the inverse of $(b_t-\lambda)^{n}$. We demonstrate this key fact --- that
$1/(z-\lambda)^{n}$ belongs to the space
$\mathscr{A}_{t}=\mathrm{Image}(\mathscr{G}_{t})$ ---
by explicitly constructing the preimage of $1/(z-\lambda)^{n}$ in
$L^{2}(\partial\mathbb{D},\nu_{t})$. Specifically, using results from
\cite{BianeJFA} or \cite{DHKLargeN} about the generating function of the
transform $\mathscr{G}_{t}$, we will show that, for all
$\omega\in\mathrm{supp}(\nu_{t})\subset \partial\mathbb{D}$,
\[
(\mathscr{G}_{t})^{-1}\left(\frac{1}{(\,\cdot\,-\lambda)^{n}}\right)(\omega)  =\frac
{1}{(n-1)!}\left(  \frac{\partial}{\partial\lambda}\right)  ^{n-1}\left[
\frac{f_{t}(\lambda)}{\lambda}\frac{1}{\omega-f_{t}(\lambda)}\right].
\]
Recall from Section \ref{domains.sec} that $f_{t}$ maps the complement of
$\overline{\Sigma}_{t}$ to the complement of $\mathrm{supp}\,\nu_{t}$. It follows
that the function on the right-hand side is bounded---and therefore square
integrable---on $\mathrm{supp}\,\nu_{t}$, for all $\lambda\in\mathbb{C}%
\setminus\overline{\Sigma}_{t}$.  When $t=4$, the proof is very similar, except that
now we have to bypass the space $\mathscr{A}_{t}$ and go directly from
$L^{2}(\partial\mathbb{D},\nu_{t})$ to $L_{\mathrm{hol}}^{2}(b_{t},\tau)$.

The $n=2$ case of Theorem \ref{L2Inverse.thm} shows that Theorem
\ref{BrownAndL2Spec.thm} applies, and we conclude that the Brown measure of
$b_{t}$ is supported in $\overline{\Sigma}_{t}$.

\section{The two-parameter case\label{twoParameter.sec}}

In this section, we discuss the generalization of the process $b_t$ and the Segal--Bargmann transform
to the two-parameter setting $b_{s,t}$ of \cite{DHKLargeN,Ho,KempLargeN}; since $b_t = b_{t,t}$, we will prove the
single-time theorems as stated as special cases of the general two-parameter framework.
We mostly follow the notation in \cite[Section 2.5]{Ho}.

\subsection{Brownian motions}

Fix positive real numbers $s$ and $t$ with $s>t/2$. Let $\{x_{r}\}_{r\ge0}$ and
$\{y_{r}\}_{r\ge0}$ be freely independent free additive Brownian motions in a tracial von Neumann algebra 
$(\mathcal{B},\tau)$, with time-parameter
denoted by $r$. Now define%
\[
w_{s,t}(r)=\sqrt{s-\frac{t}{2}}~x_{r}+i\sqrt{\frac{t}{2}}~y_{r},
\]
which we call a free elliptic $(s,t)$ Brownian motion. The particular
dependence of the coefficients on $s$ and $t$ is chosen to match the
two-parameter Segal--Bargmann transform, which will be discussed below.
Note: when $s=t$,
$$w_{t,t}(r) = \sqrt{\frac{t}{2}}(x_r+iy_r) = \sqrt{t}c_r$$
in terms of the free circular Brownian $c_r$ motion in Section \ref{section SDEs}.

We now define a \textquotedblleft free multiplicative $(s,t)$ Brownian
motion\textquotedblright\ $b_{s,t}(r)$ as a solution to the free stochastic
differential equation%
\begin{equation}
db_{s,t}(r)=i~b_{s,t}(r)~dw_{s,t}(r)-\frac{1}{2}(s-t)b_{s,t}(r)\,dr\label{bstSDE}%
\end{equation}
subject to the initial condition $b_{s,t}(0)=1$. (The second term on the
right-hand side of \eqref{bstSDE} is an It\^{o} correction term that can be
eliminated by writing the equation as a Stratonovich SDE.) We also use the
notation%
\begin{equation}
b_{s,t}=b_{s,t}(1).\label{bstDef}%
\end{equation}
When $s=t$, \eqref{bstSDE}
becomes
\[ db_{t,t}(r) = b_{t,t}(r)\,i\sqrt{t}\, dc_r. \]
Using the fact (from the usual Brownian scaling and rotational invariance) that the process $i\sqrt{t}c_r$
has the same law as the process $c_{tr}$, we see that $b_{t,t} = b_{t,t}(1)$ has the same noncommutative
distribution as the free multiplicative Brownian motion $b_{t}$.  On the other hand, the
limiting case $(s,t)=(1,0)$ gives a free unitary Brownian motion $b_{1,0}(r) = u_r$, cf.\ \eqref{e.free.SDEs}.

We can regard $b_{s,t}(r)$ as the large-$N$ limit of a certain Brownian motion
on the general linear group $\mathrm{GL}(N;\mathbb{C})$ as follows. We define an inner
product $\left\langle \cdot,\cdot\right\rangle _{s,t}$ on the Lie algebra
$\mathfrak{gl}(N;\mathbb{C})$ by%
\[
\left\langle X_{1}+iY_{1},X_{2}+iY_{2}\right\rangle _{s,t}=\frac{N}%
{\sqrt{s-t/2}}\left\langle X_{1},X_{2}\right\rangle +\frac{N}{\sqrt{t/2}%
}\left\langle Y_{1,}Y_{2}\right\rangle ,
\]
where $X_{1}$, $X_{2}$, $Y_{1}$, and $Y_{2}$ are in the Lie algebra $\mathfrak{u}(N)$ of
$\mathrm{U}(N)$ and where the inner products on the right-hand side are the standard
Hilbert--Schmidt inner product $\langle X,Y\rangle = \mathrm{Trace}(Y^\ast X)$. We extend
this inner product to a left-invariant Riemannian metric on $\mathrm{GL}(N;\mathbb{C})$ and we then let%
\[
B_{s,t}^{N}(r)
\]
be the Brownian motion with respect to this metric.  In \cite{KempLargeN}, the second author
showed that the process $B_{s,t}^{N}(\cdot)$ converges (in the sense of Definition \ref{def.limit.process})
 to the process $b_{s,t}(\cdot)$, for all positive real numbers $s$ and $t$ with $s>t/2$.
(We are translating the results of \cite{KempLargeN} into the parametrizations used in \cite{Ho}.)

\subsection{Segal--Bargmann transform}

Meanwhile, the first author and Driver introduced in \cite{DH} a
\textquotedblleft two-parameter\textquotedblright\ Segal--Bargmann transform;
see also \cite{Ha1999}. In the case of the unitary group $\mathrm{U}(N)$, the transform
is a unitary map
\[
\mathscr{G}_{s,t}^{N}:L^{2}(\mathrm{U}(N),\rho_{s})\rightarrow\mathcal{H}L^{2}(\mathrm{GL}(N;\mathbb{C}%
),\gamma_{s,t}),
\]
where $\rho_{s}$ is the same heat kernel measure as in the one-parameter
transform, but evaluated at time $s$, and where $\gamma_{s,t}$ is a heat
kernel measure on $\mathrm{GL}(N;\mathbb{C})$. Specifically, $\gamma_{s,t}$ is the
distribution of the Brownian motion $B_{s,t}^{N}(r)$ at $r=1$. The transform
itself is defined precisely as in the one-parameter case:%
\[
\mathscr{G}_{s,t}^{N}f=(e^{t\Delta/2}f)_{\mathbb{C}}\text{;}%
\]
only the inner products on the domain and range have changed. When $s=t$, the
transform $\mathscr{G}_{s,t}^{N}$ coincides with the one-parameter transform
$\mathscr{G}_{t}^{N}$.

In \cite{DHKLargeN}, the authors and Driver showed that the transform
$\mathscr{G}_{s,t}^{N}$ has limiting properties as $N\rightarrow\infty$ similar to those
of $\mathscr{G}_{t}^{N}$. Specifically, for each Laurent polynomial $p$ in one variable,
we showed that there is a unique Laurent polynomial $q_{s,t}$ in one variable
such that (abusing notation slightly)
\[
\mathscr{G}_{s,t}^{N}(p(U))=q_{s,t}(Z)+O(1/N^{2}),\quad Z\in \mathrm{GL}(N;\mathbb{C}).
\]
As an example, if $\,p(u)=u^{2}$, then $q_{s,t}(z)=e^{-t}(z^{2}-te^{-(s-t)/2}%
z)$, so that the transform of the matrix-valued function $F\colon U\mapsto U^{2}$ on
$\mathrm{U}(N)$ satisfies
\[
(\mathscr{G}_{s,t}^{N}F)(Z)=e^{-t}(Z^{2}-te^{-(s-t)/2}Z)+O(1/N^{2}),\quad Z\in
\mathrm{GL}(N;\mathbb{C}).
\]
(See \cite[p.\ 2592]{DHKLargeN}.)

In \cite{Ho}, Ho then constructed an integral transform $\mathscr{G}_{s,t}$
mapping $L^{2}(\partial\mathbb{D},\nu_{s})$ into a space of holomorphic functions on a
certain domain $\Sigma_{s,t}$ in the plane. Ho's transform $\mathscr{G}_{s,t}$
is uniquely determined by the fact that%
\[
\mathscr{G}_{s,t}(p)=q_{s,t}%
\]
for all Laurent polynomials $p$. Ho gave a description of $\mathscr{G}_{s,t}$
in terms of free probability similar to the description of Biane's transform
$\mathscr{G}_{t}$ given in Section \ref{freeSBT.sec}, and he proved a unitary
isomorphism theorem similar to Biane's result described in Theorem
\ref{bianeSBT.thm}.

\subsection{The domains $\Sigma_{s,t}$}

Ho's domains have the property that\textit{ }$f_{s-t}$\textit{ maps the
complement of }$\Sigma_{s}$\textit{ to the complement of }$\Sigma_{s,t}$. That
is to say, $\Sigma_{s,t}$ is the complement of $f_{s-t}(\mathbb{C}\setminus\Sigma_{s})$:%
\begin{equation}
\Sigma_{s,t}=\mathbb{C}\setminus f_{s-t}(\mathbb{C}\setminus\Sigma_{s})).
\label{SigmastDefn}%
\end{equation}
(See Figure \ref{two_stplots.fig} along with  \cite[Figures 2 and 3]{Ho}.)
Note that $\Sigma_{t,t}$ is the same as $\Sigma_{t}$. The topology of the
domain $\Sigma_{s,t}$ is determined by $s$; it is simply connected for
$s\leq4$ and doubly connected for $s>4$.%

\begin{figure}[htpb]%
\centering
\includegraphics[scale=0.6]
{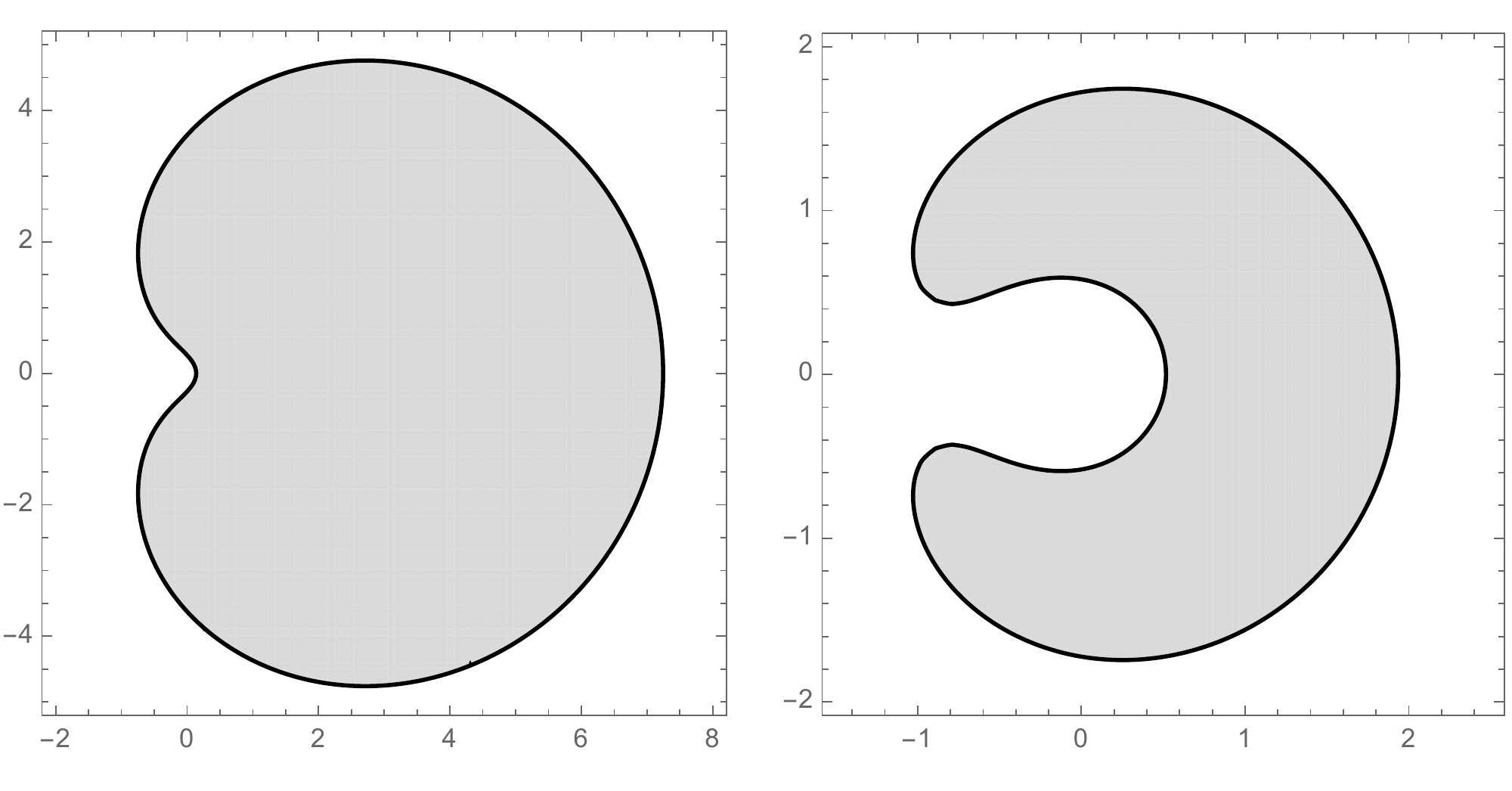}%
\caption{The domain $\Sigma_{s}$ with $s=3$ (left) and the domain
$\Sigma_{s,t}$ with $s=3$, $t=1$ (right). The map $f_{s-t}=f_{2}$ takes the
complement of the domain on the left to the complement of the domain on the
right}%
\label{two_stplots.fig}%
\end{figure}

We need a two-parameter version of Proposition \ref{ftMaps.prop}. To formulate
the correct generalization, we first note that the function $f_{s}$ satisfies%
\[
f_{s}(0)=0;\quad f_{s}^{\prime}(0)=e^{s/2}\neq0.
\]
Thus, $f_{s}$ has a local inverse defined near zero, which we denote by
$\chi_{s}$. Recall from \eqref{e.supp.nut} that the support of the
measure $\nu_{s}$ is a proper arc inside the unit circle for $s<4$ and the
whole unit circle for $s\geq4$.

\begin{proposition}
For all $s>0$, $\chi_{s}$ can be extended uniquely to a
holomorphic function on $\mathbb{C}\setminus\mathrm{supp}\,\nu_{s}$ satisfying%
\begin{equation}
\chi_{s}(1/z)=1/\chi_{s}(z). \label{chiProperty}%
\end{equation}

\end{proposition}

Note that when $s\geq4$, the support of $\nu_{s}$ is the entire unit circle,
in which case $\mathbb{C}\setminus\mathrm{supp}\,\nu_{s}$ is a disconnected
set. For such values of $s$, the proposition is really asserting just that
$\chi_{s}$ extends from a neighborhood of the origin to the open unit disk, at
which point \eqref{chiProperty} serves to \textit{define} $\chi_{s}(z)$ for
$\left\vert z\right\vert >1$.

For $s>t/2$, define a function $\chi_{s,t}$ by%
\begin{equation}
\chi_{s,t}=f_{s-t}\circ\chi_{s}. \label{chistDef}%
\end{equation}
Since $\chi_{s}$ maps $\mathbb{C}\setminus\mathrm{supp}\,\nu_{s}$
holomorphically to a region that does not include 1, we see that $\chi_{s,t}$
can also be defined holomorphically on $\mathbb{C}\setminus\mathrm{supp}\,\nu_{s}$.
Ho established the following result, generalizing Proposition
\ref{ftMaps.prop}. (See \cite[Section 4.2]{Ho}, including the discussion
following Remark 4.7.)

\begin{proposition}
\label{twoParamDomains.prop}For all positive numbers $s$ and $t$ with $s>t/2$,
define $\Sigma_{s,t}$ by \eqref{SigmastDefn}. Then the function $\chi_{s,t}$
maps $\mathbb{C}\setminus\mathrm{supp}(\nu_{s})$ injectively onto the
complement of $\overline{\Sigma}_{s,t}$. We denote the inverse function by
$f_{s,t}$, so that
\[
f_{s,t}\colon\mathbb{C}\setminus\overline{\Sigma}_{s,t}\rightarrow\mathbb{C}%
\setminus\mathrm{supp}\,\nu_{s}.
\]

\end{proposition}

Note that, at least for sufficiently small $z$, we have $f_{s,t}(z)=f_{s}(\chi_{s-t}(z))$,
by taking inverses in \eqref{chistDef}.

\subsection{The main result}

We are now ready to state the two-parameter version of Theorem \ref{main.thm}.

\begin{theorem}
\label{twoParam.thm}Let $b_{s,t}$ be the free multiplicative Brownian motion
with parameters $(s,t)$, as in \eqref{bstDef}. Then the support of the Brown
measure $\mu_{b_{s,t}}$ is contained in $\overline{\Sigma}_{s,t}$.
\end{theorem}

As in the one-parameter case, we expect that the limiting empirical eigenvalue distribution for $B_{s,t}^N:=B_{s,t}^N(1)$ will also be supported in $\Sigma_{s,t}$. This is supported by simulations; see Figure \ref{2stevalplots.fig}.  More generally,
we expect that the empirical eigenvalue distribution of the Brownian motion $B_{s,t}^{N}$ in $\mathrm{GL}(N;\mathbb{C})$ 
will converge almost surely to $\mu_{b_{s,t}}$ as $N\rightarrow\infty$.  This question will be explored in a future paper.

\begin{remark} While they do not determine the Brown measure, it is worth noting that the values
of the holomorphic moments of $b_{s,t}$ were computed in \cite{KempLargeN}:
\begin{equation}
\tau(b_{s,t}^{n})=\nu_{n}(s-t)\label{bstMoments}%
\end{equation}
for all $n\in\mathbb{Z}$ and all $s>t/2>0$; here $\nu_n(r)$ are the moments of $u_r$, cf.\ \eqref{nutMoments}.
In particular, when $s=t$, since $\nu_n(0)=1$ for all $n$, this recovers the fact that all holomorphic moments of $b_t$
are $1$. \end{remark}

\begin{figure}[htpb]%
\centering
\includegraphics[scale=0.6]
{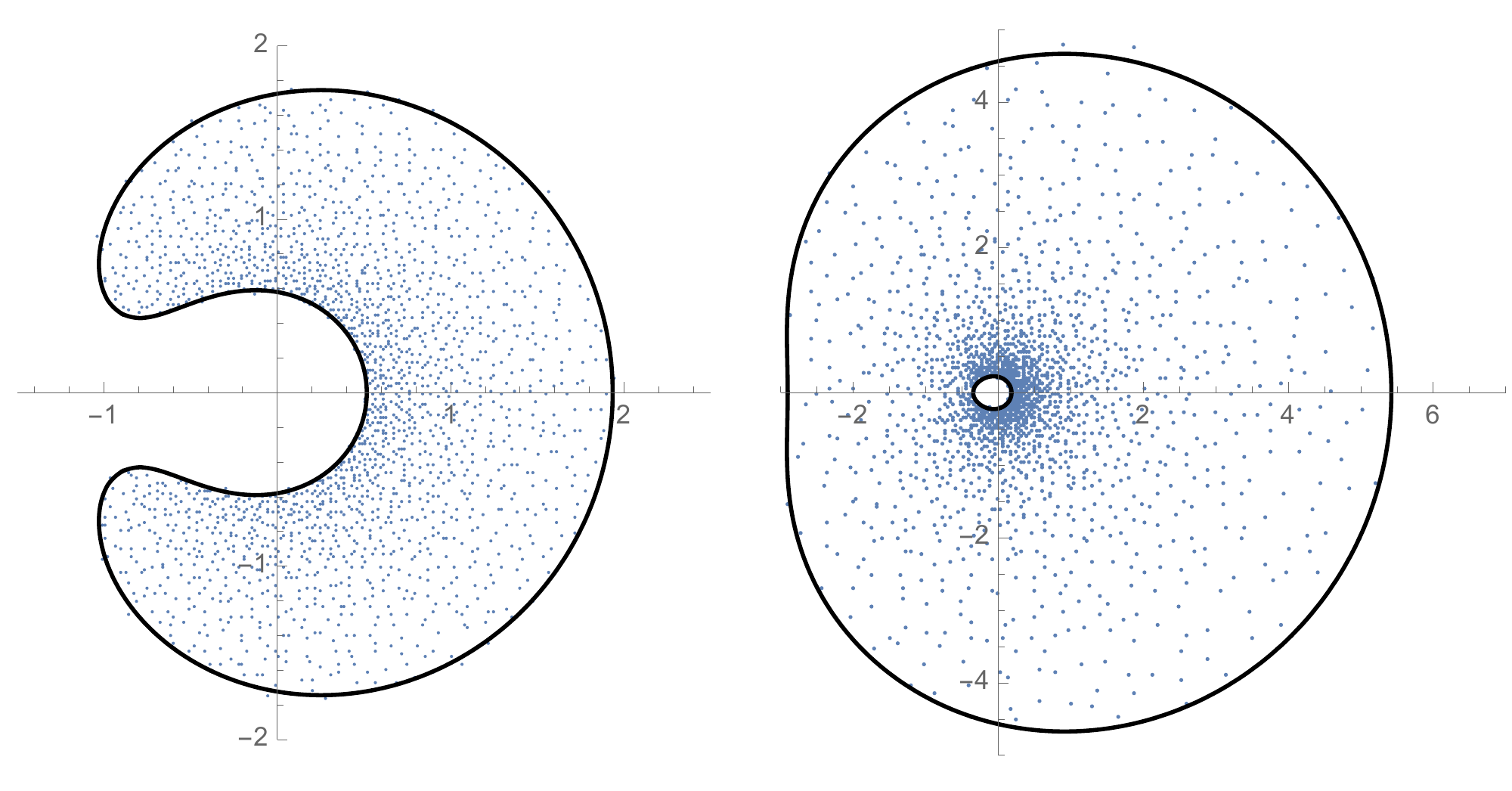}%
\caption{Simulations of the Brownian motions $B_{s,t}^N$ for $N=2000$ and $(s,t)=(3,1)$ (left) and $(s,t)=(5,3)$ (right), plotted against the domains $\Sigma_{s,t}$}%
\label{2stevalplots.fig}%
\end{figure}

\begin{remark} \label{remark.physics} Now having fully defined the two-paramer
Brownian motion and associated notation, we show how to connect the
formulas in the physics papers \cite{Nowak,Lohmayer} to our formulas.
In \cite{Nowak}, the boundary of the domain is given by Eq. (83),
which is easily seen to be equivalent to our condition for the boundary of $%
\Sigma _{t},$ namely $\left\vert f_{t}(\lambda )\right\vert =1$ with $%
\left\vert \lambda \right\vert \neq 1.$ In \cite{Lohmayer}, the boundary of
the domain in the one-parameter case is given by Eq. (5.31), which agrees
with Eq. (83) in \cite{Nowak}. In \cite{Lohmayer}, meanwhile, the boundary
of the domain in the two-parameter case is given by Eq. (6.37), which is
equivalent to saying that their variable $\hat{z}$ should belong to the
boundary of the domain $\Sigma _{t_{+}}.$ Meanwhile, using Eqs. (6.30) and
(6.36) and a bit of algebra, we find that $\hat{z}$ is related to $z$ by the
equation%
\[
z=f_{t_{-}}(\hat{z}). 
\]%
We conclude that the boundary of their domain is computed as $%
f_{t_{-}}(\partial \Sigma _{t_{+}}).$ This agrees with the boundary of our
domain $\Sigma _{s,t}$, provided we identify their parameters $t_{+}$ and $%
t_{-}$ with our $s$ and $s-t,$ respectively.

\end{remark}

\section{Proofs\label{proofs.sec}}

In this final section, we present the complete proof of Theorem \ref{twoParam.thm},
which includes the main Theorem \ref{main.thm} as the special case $s=t$.  We follow
the outline of Section \ref{outline.sec}, and will therefore provide proofs of Theorem
\ref{BrownAndL2Spec.thm} and (a two-parameter generalization of) Theorem \ref{L2Inverse.thm}.
We begin with the former.

\subsection{A general result on the support of the Brown measure\label{section L22}}

In this subsection, we work with a general operator in a tracial von Neumann algebra $(\mathcal{A},\tau)$.

For $1\leq p<\infty$, the noncommutative $L^{p}$ norm on $\mathcal{A}$ is
\[
\left\Vert a\right\Vert _{p}=(\tau\lbrack\left\vert a\right\vert ^{p}])^{1/p}.
\]
The {\bf noncommutative $L^{p}$ space} $L^{p}(\mathcal{A},\tau)$ is the completion of $\mathcal{A}$
with respect to this norm; it can be concretely realized as a space of (largely unbounded, densely-defined) operators
affiliated to $\mathcal{A}$.

For $a,b\in\mathcal{A}$ we have the inequality%
\[
\left\Vert ab\right\Vert _{p}\leq\left\Vert a\right\Vert \left\Vert
b\right\Vert _{p}.
\]
This shows that the operation of \textquotedblleft left multiplication by $a$%
\textquotedblright\ is bounded as an operator on $\mathcal{A}$ with respect to the
noncommutative $L^{p}$ norm. Thus, by the bounded linear transformation
theorem (e.g.\ \cite[Theorem I.7]{RS}), left multiplication by $a$ extends
uniquely to a bounded linear map (with the same norm) of $L^{p}(\mathcal{A}%
,\tau)$ to itself. We may easily verify that%
\begin{equation}
(ab)c=a(bc)\label{associativity1}%
\end{equation}
for $a,b\in\mathcal{A}$ and $c\in L^{p}(\mathcal{A},\tau)$; this result holds
when $c\in\mathcal{A}$ and then extends by continuity. Similar results hold
for right multiplication by $a.$ 

Similarly, since%
\begin{equation}
\left\Vert ab\right\Vert _{1}\leq\left\Vert a\right\Vert _{2}\left\Vert
b\right\Vert _{2},\label{L2L1}%
\end{equation}
for all $a,b\in\mathcal{A},$ the product map $(a,b)\mapsto ab$ can be extended
by continuity first in $a$ and then in $b,$ giving a map from~$L^{2}%
(\mathcal{A},\tau)\times L^{2}(\mathcal{A},\tau)\rightarrow L^{1}%
(\mathcal{A},\tau)$ satisfying \eqref{L2L1}. We then observe a simple
\textquotedblleft associativity\textquotedblright\ property for the actions of
$\mathcal{A}$ on $L^{1}$ and $L^{2}$:%
\begin{equation}
a(bc)=(ab)c;\quad(bc)a=b(ca).\label{associativity2}%
\end{equation}

For any $1\leq p<\infty$, we say that
an element $a\in\mathcal{A}$ \textbf{has an inverse in} $L^{p}$ if there is an
element $b$ of the noncommutative $L^{p}$ space $L^{p}(\mathcal{A},\tau)$ such
that $ab=ba=1$.

\begin{definition}
\label{LpnSpec.def} Let $a\in\mathcal{A}$, $n\in\mathbb{N}$, and $p\ge 1$.
We say that $\lambda_{0}$ belongs to the $L_{n}^{p}%
$\textbf{-resolvent set} of $a$ if $(a-\lambda)^{n}$ has an inverse, denoted
$(a-\lambda)^{-n}$, in $L^{p}$ for all $\lambda$ in a neighborhood of
$\lambda_{0}$ and $\left\Vert (a-\lambda)^{-n}\right\Vert _{L^{p}}$ is bounded
near $\lambda_{0}$. We say that $\lambda_{0}$ is in the $L_{n}^{p}%
$\textbf{-spectrum of }$a$ if $\lambda_{0}$ is not in the $L_{n}^{p}%
$-resolvent set of $a$. We denote the $L_{n}^{p}$-spectrum of $a$ by
$\mathrm{spec}_{n}^{p}(a)$.
\end{definition}

Note that if $a-\lambda_{0}$ has a bounded inverse, then so does $a-\lambda$
for all $\lambda$ sufficiently near $\lambda_{0}$, and $\left\Vert
(a-\lambda)^{-1}\right\Vert _{\mathcal{A}}$---and therefore $\left\Vert
(a-\lambda)^{-n}\right\Vert _{\mathcal{A}}$---is automatically bounded near
$\lambda_{0}$. It follows that any point in the ordinary resolvent set of $a$
is also in the $L_{n}^{p}$-resolvent set; equivalently,%
\[
\mathrm{spec}_{n}^{p}(a)\subset\mathrm{spec}(a),
\]
where $\mathrm{spec}(a)$ is the ordinary spectrum of $a$.  Theorem \ref{BrownAndL2Spec.thm} may then be restated as follows, strengthening the
standard result that the Brown measure is supported on $\mathrm{spec}(a)$.

\begin{theorem}
\label{L22support.thm}The support of the Brown measure $\mu_{a}$ of $a$
is contained in $\mathrm{spec}_{2}^{2}(a)$.
\end{theorem}

The concept of ``$L^2$ spectrum'' has come up in prior literature, most notably in the previously mentioned paper \cite{HaagerupLarsen} of Haagerup and Larsen on the Brown measures of $\mathscr{R}$-diagonal operators.  As noted at the end of Section \ref{section Brown measure}, if $a$ is $\mathscr{R}$-diagonal, its Brown measure is rotationally invariant. What is more, Haagerup and Larsen show (1) that the support of the Brown measure of $a$ is the closed disk of radius of $\|a\|_2$, if $a$ \textit{does not} have an inverse in $L^2$, and (2) that the support of the Brown measure of $a$ is the closed annulus with inner radius $1/\|a^{-1}\|_2$ and outer radius $\|a\|_2$, if $a$ \textit{does} have an inverse in $L^2$. 

Thus, in the notation introduced above, for an $\mathscr{R}$-diagonal element $a$, the point $0$ belongs to the support of the Brown measure if and only if 0 is in the $L^2_1$-spectrum of $a$. It is, at first, surprising that Haagerup and Larsen's result is for the $L^2_1$-spectrum, whereas Theorem \ref{L22support.thm} is for the $L^2_2$-spectrum. There is, however, a simple explanation for this apparent discrepancy.  In the case that $a$ is $\mathscr{R}$-diagonal, the restricted form of the free cumulants (cf.\ \cite[Lecture 15]{NS}) implies that $\|a^{-2}\|_2 = \|a^{-1}\|_2^2$; hence $0\in\mathrm{spec}^2_1(a)$ if and only if $0\in\mathrm{spec}_2^2(a)$. 

Thus, our condition in Theorem \ref{L22support.thm} is closely related to the one in \cite{HaagerupLarsen}. Indeed, Theorem \ref{L22support.thm} can be thought of as an extension of the line of reasoning in \cite{HaagerupLarsen} to general, non-$\mathscr{R}$-diagonal operators. 

It is natural to wonder, from the above definitions, how far the $L^p_n$-spectra of an operator may differ from each other, and from the actual spectrum.  To this end, Haagerup and Larsen give examples where $\mathrm{spec}_1^2(a) \ne \mathrm{spec}(a)$.  Let $h$ be a positive semi-definite operator that is not invertible in $\mathcal{A}$ but {\em is} invertible in $L^2(\mathcal{A},\tau)$ (i.e.\ its distribution $\mu_h$ has mass in every neighborhood of $0$, but $\int x^{-2}\mu_h(dx) < \infty$).  If $u$ is a Haar unitary operator freely independent from $h$, then $a=uh$ is an $\mathscr{R}$-diagonal operator for which $a^{-1}\in L^2(\mathcal{A},\tau)\setminus\mathcal{A}$; in other words $0\in\mathrm{spec}(a)\setminus\mathrm{spec}_1^2(a)$, so $\mathrm{spec}_1^2(a)\subsetneq\mathrm{spec}(a)$.  What's more, in this case, $\mathrm{spec}(a)$ is the full closed disk of radius $\|a\|_2$, while $\mathrm{spec}_1^2(a)$ is the afore-mentioned annulus, cf.\ \cite[Proposition 4.6]{HaagerupLarsen}.  Hence, the support of the Brown measure, and more generally the sets $\mathrm{spec}_n^p(a)$, may be substantially smaller than the spectrum of $a$.

To prove Theorem \ref{L22support.thm}, we need is the following.

\begin{proposition}
\label{epsilonToZero.prop}Suppose $a\in\mathcal{A}$ and $a^{2}$ has an
inverse, denoted $a^{-2}$, in $L^{2}(\mathcal{A},\tau)$. Then for all
$\varepsilon>0$, we have%
\begin{equation}
\tau\lbrack(a^{\ast}a+\varepsilon)^{-1}(aa^{\ast}+\varepsilon)^{-1}%
]\leq\left\Vert a^{-2}\right\Vert _{L^{2}(\mathcal{A},\tau)}^{2}.
\label{epsilonIneq}%
\end{equation}
\end{proposition}

Theorem \ref{L22support.thm} follows immediately from Proposition \ref{epsilonToZero.prop} together
with Corollary \ref{BrownSupport.cor}. To prove the proposition, we need the following lemmas.

\begin{lemma}
\label{inverseProperties.lem}For $a,b\in\mathcal{A},$ if $a$ is invertible in
$\mathcal{A}$ and $b$ is invertible in $L^{p},$ then $ab$ and $ba$ are
invertible in $L^{p}$ with inverses $b^{-1}a^{-1}$ and $a^{-1}b^{-1},$
respectively. If $a$ and $b$ are invertible in $L^{2}$ then $ab$ is invertible
in $L^{1}$ with inverse $b^{-1}a^{-1}.$ Finally, if $a$ is invertible
in~$L^{p}$ then $a^{\ast}$ is invertible in $L^{p}$ with inverse
$(a^{-1})^{\ast}.$ 
\end{lemma}

\begin{proof}
For $a,b\in\mathcal{A}$ with $b$ invertible in $L^{p},$ we have%
\begin{align*}
b^{-1}a^{-1}(ab)  & =((b^{-1}a^{-1})a)b=((b^{-1}(a^{-1}a))b\\
& =b^{-1}b=1,
\end{align*}
where we have used \eqref{associativity1} twice, and similarly for the product
in the other order. A similar argument, using both \eqref{associativity1} and
\eqref{associativity2}, verifies the second claim. Finally, the identity
$(ab)^{\ast}=b^{\ast}a^{\ast}$, which holds initially for $a,b\in\mathcal{A}$,
extends by continuity to the case $a\in\mathcal{A},$ $b\in L^{p}%
(\mathcal{A},\tau).$ Thus, if $a$ is invertible in $L^{p},$ then $a^{\ast
}(a^{-1})^{\ast}=(a^{-1}a)^{\ast}=1$ and similarly for the product in the
reverse order. 
\end{proof}

\begin{lemma}\label{posInv.lem}
Let $x$ be a non-negative element of $\mathcal A$ and suppose $x$ has an inverse in $L^1$. Then
\begin{equation*}
\lim_{\varepsilon\rightarrow 0^+}\tau[(x+\varepsilon)^{-1}]=\tau[x^{-1}].
\end{equation*}
\end{lemma}

\begin{proof}
We begin by noting that
\begin{align*}
x^{-1}-(x+\varepsilon)^{-1}&=x^{-1}((x+\varepsilon)-x)(x+\varepsilon)^{-1}\\
&=\varepsilon x^{-1}(x+\varepsilon)^{-1}.
\end{align*}
Now, $\Vert (x+\varepsilon)^{-1}\Vert$ is at most $1/\varepsilon$, by the equality of norm and spectral radius for self-adjoint elements. Thus,
\begin{equation*}
\Vert x^{-1}-(x+\varepsilon)^{-1}\Vert_1 \leq\varepsilon \Vert (x+\varepsilon)^{-1}\Vert\Vert x^{-1}\Vert_1 \leq\Vert x^{-1}\Vert_1 .
\end{equation*}
It follows that $\Vert (x+\varepsilon)^{-1}\Vert_1 \leq 2\Vert x^{-1}\Vert_1$ for all $\varepsilon>0$.

Let $E^x$ be the projection-valued measure associated to $x$ by the spectral theorem. Then 
\begin{equation*}
\Vert (x+\varepsilon)^{-1}\Vert_1 = \int_0^\infty \frac{1}{\lambda+\varepsilon}\,\tau[E^x(d\lambda)]
\leq 2\Vert x^{-1}\Vert_1 .
\end{equation*}
Thus, by monotone convergence, we have
\begin{equation*}
 \int_0^\infty \frac{1}{\lambda}\,\tau[E^x(d\lambda)]\leq 2\Vert x^{-1}\Vert_1 <\infty .
\end{equation*}
Once this is established, we note that for $\lambda>0$,
\begin{equation*}
\left\vert\frac{1}{\lambda}-\frac{1}{\lambda+\varepsilon}\right\vert =\frac{\varepsilon}{\lambda(\lambda+\varepsilon)}\leq \frac{1}{\lambda}.
\end{equation*}
Thus, by dominated convergence, $1/(\lambda+\varepsilon)$ converges in $L^1((0,\infty),\tau\circ E^x)$ to $1/\lambda$ as $\varepsilon\rightarrow 0$. It follows that for any sequence $\{\varepsilon_n\}$ tending to zero, the operators $(x+\varepsilon_n)^{-1}$ form a Cauchy sequence with respect to the noncommutative $L^1$ norm. Since, by dominated convergence, the functions $\lambda/(\lambda+\varepsilon_n)$ converge to 1 in $L^1((0,\infty),\tau\circ E^x)$, we can easily see that the limit of $(x+\varepsilon_n)^{-1}$ in the Banach space $L^1(\mathcal A,\tau)$ is the inverse in $L^1$ of $x$.

We have shown, then, that the (unique) inverse in $L^1$ of $x$ is the limit in $L^1$ of $(x+\varepsilon_n)^{-1}$. Applying $\tau$ to this result gives the claim, along any sequence $\{\varepsilon_n\}$ tending to zero.
\end{proof}

\begin{lemma} \label{l.inv.monotone} Let $x,y\ge 0$ be positive semidefinite operators in $\mathcal{A}$, and suppose that they are invertible in $L^1(\mathcal{A},\tau)$.  If $x\le y$ (i.e.,\ $x-y$ is positive semidefinite) then $\tau(y^{-1})\le\tau(x^{-1})$. \end{lemma}

\begin{proof} 
For all $\varepsilon >0$, the elements $x+\varepsilon$ and $y+\varepsilon$ are invertible in $\mathcal A$ and $x+\varepsilon\geq y+\varepsilon$. Thus, by the (reverse) operator monotonicity of the inverse \cite[Prop. V.1.6]{Bh}, we have 
\begin{equation}
\tau((y+\varepsilon)^{-1})\geq \tau((x+\varepsilon)^{-1}). 
\end{equation}
By Lemma \ref{posInv.lem}, the claimed result then follows.
\end{proof}

\begin{lemma} \label{l.aa*.a*a.L1.inv} Let $a\in\mathcal{A}$ and suppose that $a^2$ has an inverse in $L^2(\mathcal{A},\tau)$.  Then $a$ and $a^\ast$ have inverses in $L^2$; and hence $a^\ast a$ and $aa^\ast$ have inverses in $L^1(\mathcal{A},\tau)$.  \end{lemma}

\begin{proof} Let $b=a^{-2}$ denote the $L^2(\mathcal{A},\tau)$ inverse of $a^2$.  First, note that $a(ab) = a^2b = 1$ and $(ba)a = ba^2 =1$.  Since $ab$ and $ba$ are in $L^2(\mathcal{A},\tau)$, it follows that $a$ is both left and right invertible in $L^2(\mathcal{A},\tau)$. Hence, $a$ is invertible in $L^2(\mathcal{A},\tau)$, by the usual argument and \eqref{associativity1}. Taking adjoints shows that $a^\ast$ is also invertible in $L^2(\mathcal{A},\tau)$.  Thus $aa^\ast[(a^\ast)^{-1}a^{-1}] = 1 = [(a^\ast)^{-1}a^{-1}]aa^\ast$.  Since $(a^\ast)^{-1}$ and $a^{-1}$ are in $L^2(\mathcal{A},\tau)$, their product is in $L^1(\mathcal{A},\tau)$ by H\"older's inequality; this shows $aa^\ast$ is invertible in $L^1(\mathcal{A},\tau)$.  An analogous argument holds for $a^\ast a$.  \end{proof}

We are now ready for the proof of Proposition \ref{epsilonToZero.prop}.

\begin{proof}[Proof of Proposition \ref{epsilonToZero.prop}] 
We begin by arguing formally and then fill in the details. We start by noting
that%
\begin{align*}
(a^{\ast}a+\varepsilon)^{1/2}(aa^{\ast}+\varepsilon)(a^{\ast}a+\varepsilon
)^{1/2}-(a^{\ast}a+\varepsilon)^{1/2}(aa^{\ast})(a^{\ast}a+\varepsilon)^{1/2}
&  =\varepsilon(a^{\ast}a+\varepsilon)\\
&  \geq0.
\end{align*}
Thus, by Lemma 6.4 and the cyclic property of the trace, we have%
\begin{align}
\tau\lbrack(a^{\ast}a+\varepsilon)^{-1}(aa^{\ast}+\varepsilon)^{-1}] &
=\tau\lbrack(a^{\ast}a+\varepsilon)^{-1/2}(aa^{\ast}+\varepsilon)^{-1}%
(a^{\ast}a+\varepsilon)^{-1/2}]\nonumber\\
&  \leq\tau\lbrack(a^{\ast}a+\varepsilon)^{-1/2}(aa^{\ast})^{-1}(a^{\ast
}a+\varepsilon)^{-1/2}]\nonumber\\
&  =\tau\lbrack(aa^{\ast})^{-1}(a^{\ast}a+\varepsilon)^{-1}].\label{firstStep}%
\end{align}

We then use the same argument again. We note that%
\begin{equation}
a^{\ast}(a^{\ast}a+\varepsilon)a-a^{\ast}(a^{\ast}a)a=\varepsilon a^{\ast
}a\geq0,\label{secondTerms}%
\end{equation}
from which we obtain%
\begin{align}
\tau\lbrack(aa^{\ast})^{-1}(a^{\ast}a+\varepsilon)^{-1}] &  =\tau\lbrack
a^{-1}(a^{\ast}a+\varepsilon)^{-1}(a^{\ast})^{-1}]\nonumber\\
&  =\tau\lbrack(a^{\ast}(a^{\ast}a+\varepsilon)a)^{-1}]\nonumber\\
&  \leq\tau\lbrack((a^{\ast})^{2}a^{2})^{-1}]\nonumber\\
&  =\tau\lbrack a^{-2}(a^{-2})^{\ast}].\nonumber\\
&  =\left\Vert a^{-2}\right\Vert _{L^{2}(\mathcal{A},\tau)}^{2}%
\label{secondStep}%
\end{align}
Combining \eqref{firstStep} and \eqref{secondStep} gives the desired inequality.

To make the argument rigorous, we need to make sure that all the relevant
inverses exist in $L^{1}(\mathcal{A},\tau),$ so that Lemma 6.4 is applicable.
The operator $(a^{\ast}a+\varepsilon)^{1/2}(aa^{\ast}+\varepsilon)(a^{\ast
}a+\varepsilon)^{1/2}$ is invertible in $\mathcal{A}$ and thus in
$L^{1}(\mathcal{A},\tau).$ On the other hand, by assumption $a^{2}$ is
invertible in $L^{2}(\mathcal{A},\tau)$. It then follows from Lemma 6.5 that
$aa^{\ast}$ is invertible in $L^{1}(\mathcal{A},\tau),$ so that $(a^{\ast
}a+\varepsilon)^{1/2}(aa^{\ast})(a^{\ast}a+\varepsilon)^{1/2}$ is also
invertible in $L^{1}(\mathcal{A},\tau),$ by Lemma \ref{inverseProperties.lem}.
Thus, Lemma 6.4 is applicable in \eqref{firstStep}.

We now consider the two terms on the left-hand side of \eqref{secondTerms}. By
Lemma 6.5, both $a$ and $a^{\ast}$ are invertible in $L^{2}$; it then follows
from Lemma \ref{inverseProperties.lem} that $a^{\ast}(a^{\ast}a+\varepsilon)$
is invertible in $L^{2}$ and that $a^{\ast}(a^{\ast}a+\varepsilon)a$ is
invertible in $L^{1}.$ Meanwhile, $a^{2}$ is, by assumption, invertible in
$L^{2}$; it then follows from Lemma \ref{inverseProperties.lem} that
$(a^{\ast})^{2}$ is invertible in $L^{2}.$ Thus, using Lemma
\ref{inverseProperties.lem} one last time, we conclude that $(a^{\ast}%
)^{2}a^{2}$ is invertible in $L^{1}.$ Thus, Lemma 6.5 is also applicable in
\eqref{secondStep}.
\end{proof}

\subsection{Computing the $L_{2}^{2}$ spectrum\label{computeL2spec.sec}}

We consider the free multiplicative $(s,t)$ Brownian motion $b_{s,t}$ defined
in \eqref{bstDef} as an element of a tracial von Neumann algebra $(\mathcal{B},\tau)$. Recall that when $s=t$, the operator $b_{s,t}$
has the same noncommutative distribution as the ordinary free multiplicative
Brownian motion $b_{t}$ described in Section \ref{section SDEs}. Our
goal is to prove a two-parameter version of Theorem \ref{L2Inverse.thm} from
Section \ref{outline.sec}, by constructing an $L^{2}$ inverse to
$(b_{s,t}-\lambda)^{n}$ for $\lambda$ in the complement of
$\overline{\Sigma}_{s,t}$, with local bounds on the norm of the inverse. This will, in
particular, show that%
\begin{equation}
\mathrm{spec}_{n}^{2}(b_{s,t})\subset\overline{\Sigma}_{s,t} \label{specContain}%
\end{equation}
for all $n$. (Recall Definition \ref{LpnSpec.def}.) If we specialize
\eqref{specContain} to the case $n=2$, Theorem \ref{L22support.thm} will then
tell us that the support of the Brown measure of $b_{s,t}$ is contained
in $\overline{\Sigma}_{s,t}$.

Our tool is the two-parameter \textquotedblleft free Hall
transform\textquotedblright\ $\mathscr{G}_{s,t}$ (cf.\ Section
\ref{twoParameter.sec}), which includes the one-parameter transform
$\mathscr{G}_{t}=\mathscr{G}_{t,t}$ (cf.\ Section \ref{BianesTransform.sec}) as a
special case. To avoid technical issues for the case $s=4$, we introduce a
variant of the transform $\mathscr{G}_{s,t}$, denoted $\mathscr{S}_{s,t}$.
Define%
\[
L_{\mathrm{hol}}^{2}(b_{s,t},\tau)
\]
to be the closure in the noncommutative $L^{2}$ norm of the space of elements
of the form $p(b_{s,t})$, where $p$ is a Laurent polynomial in one variable.

If $\mathcal{P}$ denotes the space of Laurent polynomials in one variable,
then \cite[Theorem 1.13]{DHKLargeN} shows that $\mathscr{G}_{s,t}$ maps
$\mathcal{P}$ bijectively onto $\mathcal{P}$. We then define $\mathscr{S}%
_{s,t}$ initially on $\mathcal{P}$ by evaluating each polynomial
$\mathscr{G}_{s,t}(p)$ on the element $b_{s,t}$:%
\[
\mathscr{S}_{s,t}(p)=\mathscr{G}_{s,t}(p)(b_{s,t})\in L_{\mathrm{hol}}%
^{2}(b_{s,t},\tau).
\]
By \cite[Theorem 5.7]{Ho}, $\mathscr{S}_{s,t}$ maps $\mathcal{P}\subset
L^{2}(\partial\mathbb{D},\nu_{s})$ isometrically into $L_{\mathrm{hol}}^{2}(b_{s,t},\tau)$.
(Although some parts of the just-cited theorem implicitly assume that
$s\neq4$, Part 2 of the theorem does not depend on this assumption.)
Furthermore, since $\mathscr{G}_{s,t}$ maps \textit{onto} $\mathcal{P}$, the
image of $\mathscr{S}_{s,t}$ is dense in $L_{\mathrm{hol}}^{2}(b_{s,t},\tau)$.
Thus, $\mathscr{S}_{s,t}$ extends to a unitary map%
\[
\mathscr{S}_{s,t}:L^{2}(\partial\mathbb{D},\nu_{s})\rightarrow L_{\mathrm{hol}}^{2}%
(b_{s,t},\tau).
\]

For $\lambda\in\mathbb{C}\setminus\overline{\Sigma}_{s,t}$, we will then construct
an inverse $(b_{s,t}-\lambda)^{-n}$ to $(b_{s,t}-\lambda)^{n}$ in
$L_{\mathrm{hol}}^{2}(b_{s,t},\tau)$ by constructing the preimage of
$(b_{s,t}-\lambda)^{-n}$ under $\mathscr{S}_{s,t}$. Note that in Section
\ref{outline.sec}, we outlined the proof in the case $s=t$, with $t\neq4$. In
that case, \cite[Lemma 17]{BianeJFA} allows us to identify
$L_{\mathrm{hol}}^{2}(\mathcal{B},\tau)$ with the space $\mathscr{A}_{t}$, in which case, it
is harmless to work with $\mathscr{G}_{s,t}$ instead of $\mathscr{S}_{s,t}$.

Recall the definition of the function $f_{s,t}$ in Proposition \ref{twoParamDomains.prop}.
\begin{theorem}
\label{L2inv.thm}Define a function $r_{\lambda}^{s,t}\in L^{2}(\partial\mathbb{D},\nu_{s})$
by the formula%
\begin{equation}
r_{\lambda}^{s,t}(\omega)=\frac{f_{s,t}(\lambda)}{\lambda}\left(  \frac
{1}{\omega-f_{s,t}(\lambda)}\right)  ,\quad \omega\in\mathrm{supp}(\nu_{s}),
\label{rstFormula}%
\end{equation}
for all $\lambda\in\mathbb{C}\setminus\overline{\Sigma}_{s,t}$. Then for all such
$\lambda$, we have%
\[
\mathscr{S}_{s,t}(r_{\lambda}^{s,t})=(b_{s,t}-\lambda)^{-1}.
\]
That is to say, $\mathscr{S}_{s,t}(r_{\lambda}^{s,t})$ is an inverse in
$L^{2}$ of $(b_{s,t}-\lambda)$.
\end{theorem}

Recall from Proposition \ref{twoParamDomains.prop} that $f_{s,t}(\lambda)$
is outside the support of $\nu_{s}$ for all $\lambda$ in the complement
of $\overline{\Sigma}_{s,t}$. Thus, $r_{\lambda}^{s,t}(\omega)$ is bounded and
is therefore a $\nu_s$--square integrable function of $\omega$ for all
$\lambda\in\mathbb{C}\setminus\overline{\Sigma}_{s,t}$.

When $\lambda=0$, we interpret $r_{\lambda}^{s,t}$ to be the limit of the
right-hand side of \eqref{rstFormula} as $\lambda$ approaches zero, which is
easily computed to be $r_{0}^{s,t}(\omega)=e^{t/2}\omega^{-1}$.

\begin{proof}
For the case $\lambda=0$, we compute that 
$\mathscr{G}_{s,t}(\omega^{-1})=e^{-t/2}z^{-1}$,
as may be verified from the behavior of $\mathscr{G}_{s,t}$
with respect to inversion \cite[Eq.\ (5.2)]{DHKLargeN} and the recursive
formula in \cite[Proposition 5.2]{DHKLargeN}. Thus, by definition, we have
$\mathscr{S}_{s,t}(\omega^{-1})=e^{-t/2}b_{s,t}^{-1}$ so that $\mathscr{S}%
_{s,t}(e^{t/2}\omega^{-1})=b_{s,t}^{-1}$, which is the $\lambda=0$ case of the
theorem. The subsequent calculations assume $\lambda\neq0$.

We start by considering large values of $\lambda$. For $\left\vert
\lambda\right\vert >\left\Vert b_{s,t}\right\Vert $, the element
$b_{s,t}-\lambda$ has a \textit{bounded} inverse, which may be computed as a
power series:%
\[
(b_{s,t}-\lambda)^{-1}=-\frac{1}{\lambda}(1-b_{s,t}/\lambda)^{-1}=-\frac
{1}{\lambda}\sum_{k=0}^{\infty}\lambda^{-k}(b_{s,t})^{k},
\]
with the series converging in the operator norm and thus also in the
noncommutative $L^{2}$ norm. Applying $\mathscr{S}_{s,t}^{-1}$ term by term
gives
\begin{align}
\mathscr{S}_{s,t}^{-1}\left(  (b_{s,t}-\lambda)^{-1}\right)  (\omega) &=
-\frac{1}{\lambda}\sum_{k=0}^{\infty}\lambda^{-k}p_{k}^{s,t}(\omega)\label{InverseResolve1}\\
& = -\frac{1}{\lambda}\left[  1+\Pi(s,t,\omega,1/\lambda)\right],
\label{InverseResolve2}%
\end{align}
where $p_{k}^{s,t}$ is the unique polynomial such that
\[
\mathscr{G}_{s,t}(p_{k}^{s,t})(z)=z^{k},\quad k\in\mathbb{Z},
\]
and where $\Pi$ is the generating function in \cite[Theorem 1.17]{DHKLargeN}.

Meanwhile, according to \cite[Eq.\ (1.21)]{DHKLargeN}, we have the following implicit formula for $\Pi$:
\[
\Pi(s,t,\omega,f_{s-t}(z))=\frac{1}{1-\omega f_{s}(z)}-1.
\]
(This result extends a formula of Biane in the $s=t$ case.) Then, at least for
sufficiently small $z$, we may replace $z$ by $\chi_{s-t}(z)$, where $\chi_{s,t}$ is the inverse function to $f_{s-t}$. This substitution gives
\begin{equation}
\Pi(s,t,\omega,z)=\frac{1}{1-\omega f_{s}(\chi_{s-t}(z))}-1. \label{PiFormula}%
\end{equation}
Plugging this expression into \eqref{InverseResolve2} gives%
\[
r_{\lambda}^{s,t}(\omega)=-\frac{1}{\lambda}\frac{1}{1-\omega f_{s}(\chi_{s-t}%
(1/\lambda))},
\]
for $\lambda$ sufficiently large, which
simplifies---using \eqref{chiProperty} and the identity $f_s(1/z)=1/f_s(z)$---to
the expression in \eqref{rstFormula}.

Similarly, for $0<\left\vert \lambda\right\vert <1/\left\Vert b_{s,t}%
^{-1}\right\Vert $, we use the series expansion%
\begin{align*}
(b_{s,t}-\lambda)^{-1}  &  =b_{s,t}^{-1}(1-\lambda b_{s,t}^{-1})^{-1}\\
&  =b_{s,t}^{-1}\sum_{k=0}^{\infty}\lambda^{k}b_{s,t}^{-k}\\
&  =\frac{1}{\lambda}\sum_{k=1}^{\infty}\lambda^{k}b_{s,t}^{-k}.
\end{align*}
Since $p_{-k}^{s,t}(\omega)=p_{k}^{s,t}(\omega^{-1})$ (see \cite[Eq. (5.2)]{DHKLargeN}),
we may apply $\mathscr{S}_{s,t}^{-1}$ term by term to obtain
\begin{align*}
\mathscr{S}_{s,t}^{-1}\left(  (b_{s,t}-\lambda)^{-1}\right)  (\omega)  &  =\frac
{1}{\lambda}\sum_{k=1}^{\infty}\lambda^{k}p_{k}^{s,t}(\omega^{-1})\\
&  =\frac{1}{\lambda}\Pi(s,t,\omega^{-1},\lambda).
\end{align*}
Using the formula \eqref{PiFormula} for $\Pi$ (for sufficiently small
$\lambda$) and simplifying gives the same expression as for the large-$\lambda
$ case.

Finally, for general $\lambda\in\mathbb{C}\setminus\overline{\Sigma}_{s,t}$ we use
an analytic continuation argument. The complement of the closure of
$\Sigma_{s,t}$ has at most two connected components, the unbounded component
and, for $s\geq4$, a bounded component containing 0. Now, for all $\lambda
\in\mathbb{C}\setminus\overline{\Sigma}_{s,t}$, the function $r_{\lambda}^{s,t}$ in
\eqref{rstFormula} is a well-defined element of $L^{2}(\partial\mathbb{D},\nu_{s})$,
because (by Proposition \ref{twoParamDomains.prop}) $f_{s,t}$ maps $\mathbb{C}\setminus\overline{\Sigma}_{s,t}$ into
$\mathbb{C}\setminus\mathrm{supp}\,\nu_{s}$. Furthermore, it is evident from \eqref{rstFormula} that  $r_{\lambda}^{s,t}$
is a weakly holomorphic function of
$\lambda\in\mathbb{C}\setminus\overline{\Sigma}_{s,t}$, with values in
$L^{2}(\partial\mathbb{D},\nu_{s})$, meaning that $\lambda\mapsto \phi(r_{\lambda}^{s,t})$
is a holomorphic for each bounded linear functional on $L^{2}(\partial\mathbb{D},\nu_{s})$.

Thus, since $\mathscr{S}_{s,t}$ is unitary (hence bounded), $\mathscr{S}_{s,t}(r_{\lambda}^{s,t})$ is a weakly holomorphic function
of $\lambda\in\mathbb{C}\setminus\overline{\Sigma}_{s,t}$, with values in
$L_{\mathrm{hol}}^{2}(\mathcal{B},\tau)$. It is then easy to see that
\begin{equation}
(b_{s,t}-\lambda)\mathscr{S}_{s,t}(r_{\lambda}^{s,t})=b_{s,t}\mathscr{S}_{s,t}(r_{\lambda}^{s,t})-\lambda\mathscr{S}_{s,t}(r_{\lambda}^{s,t})
\label{holomorphicExpression}%
\end{equation}
is also weakly holomorphic. After all, applying a bounded linear functional $\phi$ gives
\begin{equation}
\phi(b_{s,t}\mathscr{S}_{s,t}(r_{\lambda}^{s,t}))-\lambda\phi(\mathscr{S}_{s,t}(r_{\lambda}^{s,t})).
\label{twoTermsHolomorphic}
\end{equation}
Since multiplication by $b_{s,t}$ is a bounded linear map on $L_{\mathrm{hol}}^{2}(\mathcal{B},\tau)$, the linear functional $\phi(b_{s,t}\cdot)$ is bounded, so both terms in \eqref{twoTermsHolomorphic} are holomorphic in $\lambda$.

Meanwhile, we
have shown that \eqref{holomorphicExpression} is equal to 1 on a nonempty,
open subset of each connected component of
$\mathbb{C}\setminus\overline{\Sigma}_{s,t}$. Since, also, \eqref{holomorphicExpression} is weakly holomorphic, it must be equal to 1 on all of $\mathbb{C}\setminus\overline{\Sigma}_{s,t}$.
\end{proof}

We emphasize that, although for very large and small $\lambda$ the standard
power-series argument gives an inverse of $b_{s,t}-\lambda$ in the algebra
$\mathcal{B}$, the analytic continuation takes place in $L_{\mathrm{hol}}%
^{2}(\mathcal{B},\tau)$, which is the range of the transform $\mathscr{S}%
_{s,t}$. Thus, for general $\lambda\in\mathbb{C}\setminus\overline{\Sigma}_{s,t}$,
we are guaranteed that $b_{s,t}-\lambda$ has an inverse in $L^{2}$, but not
necessarily in $\mathcal{B}$.

\begin{corollary}
For all $\lambda\in\mathbb{C}\setminus\overline{\Sigma}_{s,t}$ and all positive
integers $n$, the operator $(b_{s,t}-\lambda)^{n}$ has an inverse
$(b_{s,t}-\lambda)^{-n}$ in $L^{2}(\mathcal{B},\tau)$. Specifically,%
\begin{equation}
(b_{s,t}-\lambda)^{-n}=\mathscr{S}_{s,t}\left(\frac{1}{(n-1)!}\left(\frac{\partial}{\partial\lambda}\right)  ^{n-1}r_{\lambda}^{s,t}\right),
\label{InverseN}%
\end{equation}
where $r_{\lambda}^{s,t}$ is as in \eqref{rstFormula}. Furthermore,
$\left\Vert (b_{s,t}-\lambda)^{-n}\right\Vert_{L^2(\mathcal{B},\tau)} $ is locally bounded on
$\mathbb{C}\setminus\overline{\Sigma}_{s,t}$.
\end{corollary}

\begin{proof}
We let%
\[
r_{\lambda}^{s,t,n}(\omega)=\frac{1}{(n-1)!}\left(  \frac{\partial}{\partial
\lambda}\right)  ^{n-1}r_{\lambda}^{s,t}(\omega).
\]
If we inductively compute the derivatives in the definition of $r_{\lambda}^{s,t,n}(\omega)$, we will find that the result is polynomial in $1/(\omega-f_{s,t}(\lambda))$, with coefficients that are holomorphic functions of $\lambda\in\mathbb{C}\setminus\overline{\Sigma}_{s,t}$. Thus, for each $n$, the quantity $r_{\lambda}^{s,t,n}(\omega)$ is jointly continuous as a function of $\omega\in\mathrm{supp}(\nu_t)$ and $\lambda\in\mathbb{C}\setminus\overline{\Sigma}_{s,t}$. Thus, $\| r_{\lambda
}^{s,t,n}\|_{L^{2}(\partial\mathbb{D},\nu_{s})}$ is finite and depends continuously on $\lambda$. Thus, once \eqref{InverseN} is verified, the
local bounds on the norm of $(b_{s,t}-\lambda)^{-n}$ will follow from the
unitarity of $\mathscr{S}_{s,t}$.

We establish \eqref{InverseN} by induction on~$n$, the $n=1$ case being the
content of Theorem \ref{L2inv.thm}. Assume, then, the result for a fixed $n$ and recall that $r_{\lambda}^{s,t,n}(\omega)$ is a polynomial in $1/(\omega-f_{s,t}(\lambda))$, with coefficients that are holomorphic functions of $\lambda$. It is then an elmentary matter to see that for each fixed $\lambda\in\mathbb{C}\setminus\overline{\Sigma}_{s,t}$, the limit
\[
\lim_{h\rightarrow0}\frac{r_{\lambda+h}^{s,t,n}(\omega)-r_{\lambda}^{s,t,n}(\omega)}{h}%
\]
exists as a uniform limit on $\mathrm{supp}\,\nu_{s}$, and thus also in
$L^{2}(\partial\mathbb{D},\nu_{s})$. Applying $\mathscr{S}_{s,t}$ and using our induction
hypothesis gives%
\begin{equation} \label{stageN.0}
\mathscr{S}_{s,t}\left(  \frac{\partial}{\partial\lambda}r_{\lambda}%
^{s,t,n}(\omega)\right)  =\lim_{h\rightarrow0}\frac{1}{h}[(b_{s,t}-(\lambda
+h))^{-n}-(b_{s,t}-\lambda)^{-n}],
\end{equation}
where the limit is in $L_{\mathrm{hol}}^{2}(b_{s,t},\tau)$.

We now multiply both sides of \eqref{stageN.0} by $(b_{s,t}-\lambda)^{n+1}$,
which we write as $(b_{s,t}-\lambda)(b_{s,t}-\lambda)^{n}$. Since
multiplication by an element of $\mathcal{B}$ is a continuous linear map on
$L^{2}(\mathcal{B},\tau)$, we obtain%
\begin{align}
&  (b_{s,t}-\lambda)^{n+1}\mathscr{S}_{s,t}\left(  \frac{\partial}%
{\partial\lambda}r_{\lambda}^{s,t,n}(\omega)\right) \nonumber\\
&  =\lim_{h\rightarrow0}(b_{s,t}-\lambda)\frac{1}{h}[(b_{s,t}-\lambda
)^{n}(b_{s,t}-(\lambda+h))^{-n}-1]. \label{stageN}%
\end{align}
In the first term on the right-hand side of \eqref{stageN}, we write%
\begin{equation}
(b_{s,t}-\lambda)^{n}=(b_{s,t}-(\lambda+h)+h)^{n}=\sum_{k=0}^{n}\binom{n}%
{k}(b_{s,t}-(\lambda+h))^{n-k}h^{k}. \label{binomial}%
\end{equation}
Upon substituting \eqref{binomial} into \eqref{stageN}, the $k=0$ term cancels
with the existing term of $-1$, while the $k=1$ term gives%
\begin{equation}
(b_{s,t}-\lambda)\frac{1}{h}\cdot nh(b_{s,t}-(\lambda+h))^{-1}=n(b_{s,t}%
-\lambda)(b_{s,t}-(\lambda+h))^{-1}. \label{kEqualsOne}%
\end{equation}
If we again write $(b_{s,t}-\lambda)=(b_{s,t}-(\lambda+h)+h)$, we see that
\eqref{kEqualsOne} tends to $n\cdot1$ as $h\to0$. Finally, all terms
with $k\geq2$ will vanish in the limit, leaving us with%
\[
(b_{s,t}-\lambda)^{n+1}\mathscr{S}_{s,t}\left(  \frac{\partial}{\partial
\lambda}r_{\lambda}^{s,t,n}(\omega)\right)  =n\cdot1.
\]
Thus,%
\[
\mathscr{S}_{s,t}\left(  \frac{1}{n}\frac{\partial}{\partial\lambda}%
r_{\lambda}^{s,t,n}(\omega)\right)  =(b_{s,t}-\lambda)^{-(n+1)},
\]
which is just the level-$(n+1)$ case of the corollary.
\end{proof}

\subsection*{Acknowledgments} We would like to thank Adrian Ioana and Ching Wei Ho for useful conversations.  We also thank Maciej Nowak for making us aware of the papers \cite{Nowak,Lohmayer} and for many illuminating discussions about them.  Finally, we thank the anonymous referee for carefully reading our manuscript and, in particular, for suggesting additions that improved the discussion of $L^p_n$-spectrum.


\begin{thebibliography}{99}                                                                                               %
                                                                                             %
\bibitem {Bai}Z. D. Bai, Circular law, \textit{Ann. Probab.} \textbf{25}
(1997), 494--529.

\bibitem {BV}H. Bercovici and D. Voiculescu, L\'{e}vy--Hin\v{c}in type
theorems for multiplicative and additive free convolution, \textit{Pacific J.
Math.} \textbf{153} (1992), 217--248.

\bibitem {Bh}R. Bhatia, Matrix analysis. Graduate Texts in Mathematics, 169.
Springer-Verlag, New York, 1997.

\bibitem {BianeFields}P. Biane, \textquotedblleft Free Brownian motion, free
stochastic calculus and random matrices.\textquotedblright\ \textit{In}\ Free
Probability Theory (Waterloo, ON, 1995), 1--19. Fields Institute
Communications 12. Providence, RI: American Mathematical Society, 1997.

\bibitem {BianeJFA}P. Biane, Segal--Bargmann transform, functional calculus on
matrix spaces and the theory of semi-circular and circular systems, \textit{J.
Funct. Anal.}, \textbf{144} (1997), 232--286.

\bibitem {BianeSubordination}P. Biane, Processes with free increments,
\textit{Math. Z.} \textbf{227} (1998), 143--174.

\bibitem {BS1} P. Biane and R. Speicher, Stochastic calculus with respect to
free Brownian motion and analysis on Wigner space, \textit{Probab. Theory
Related Fields} \textbf{112} (1998), 373--409.

\bibitem {BS2} P. Biane and R. Speicher, Free diffusions, free entropy and
free Fisher information, \textit{Ann. Inst. H. Poincar\'{e} Probab. Statist.}
\textbf{37} (2001), 581--606.

\bibitem {Blackadar} B. Blackadar, Operator algebras.  Theory of $C^\ast$-algebras and von Neumann algebras, \textit{Encyclopedia of Mathematical Sciences} \textbf{122}.  Operator Algebras and Non-commutative Geometry, III.  Springer-Verlag, Berlin, 2006.

\bibitem {Br}Brown, L. G. Lidski\u{\i}'s theorem in the type II case. \textit{In}
Geometric methods in operator algebras (Kyoto, 1983), 1--35, Pitman Res. Notes
Math. Ser., 123, Longman Sci. Tech., Harlow, 1986.

\bibitem {Ceb}G. C\'{e}bron, Free convolution operators and free Hall
transform, \textit{J. Funct. Anal.} \textbf{265} (2013), 2645--2708.

\bibitem{CK}B. Collins and T. Kemp, Liberation of projections, \textit{J. Funct. Anal.} \textbf{266} (2014), 1988--2052.

\bibitem {DH}B. K. Driver and B. C. Hall, Yang--Mills theory and the
Segal--Bargmann transform, \textit{Comm. Math. Phys.} \textbf{201} (1999), 249--290.

\bibitem {DHKLargeN}B. K. Driver, B. C. Hall, and T. Kemp, The large-$N$ limit
of the Segal--Bargmann transform on $\mathrm{U}(N)$, \textit{J. Funct. Anal.}
\textbf{265} (2013), 2585--2644.

\bibitem {DHK-Brown-Measure}B. K. Driver, B.C. Hall, and T. Kemp, The Brown measure
of the free multiplicative Brownian motion, \textit{Preprint} arXiv:1903.11015

\bibitem {FK1}B. Fuglede and R. V. Kadison, On determinants and a property of
the trace in finite factors, \textit{Proc. Nat. Acad. Sci. U. S. A.}
\textbf{37} (1951), 425--431.

\bibitem {FK2}B. Fuglede and R. V. Kadison, Determinant theory in finite
factors, \textit{Ann. of Math. (2)} \textbf{55} (1952), 520--530.

\bibitem {Gin}J. Ginibre, Statistical ensembles of complex, quaternion, and real matrices,
\textit{J. Math. Phys.} \textbf{6} (1965), 440--449.

\bibitem {Girko}V. L. Girko, The circular law. (Russian) \textit{Teor.
Veroyatnost. i Primenen.} \textbf{29} (1984), 669--679.

\bibitem {GT}F. G\"{o}tze and A. Tikhomirov, The circular law for random
matrices, \textit{Ann. Probab.} \textbf{38} (2010), 1444--1491.

\bibitem {GM}L. Gross and P. Malliavin, Hall's transform and the
Segal--Bargmann map. \textit{In} It\^{o}'s stochastic calculus and probability
theory (N. Ikeda, S. Watanabe, M. Fukushima and H. Kunita, Eds.), 73--116,
Springer, 1996.

\bibitem{HaagerupLarsen} U. Haagerup and F. Larsen, Brown's spectral distribution measure for $R$-diagonal elements in finite von Neumann algebras, \textit{J. Funct. Anal.} \textbf{176} (2000), 331--367.

\bibitem {Ha1994}B. C. Hall, The Segal--Bargmann \textquotedblleft coherent
state\textquotedblright\ transform for compact Lie groups. \textit{J. Funct.
Anal.} \textbf{122} (1994), 103--151.

\bibitem {Ha1999}B. C. Hall, A new form of the Segal--Bargmann transform for
Lie groups of compact type, \textit{Canad. J. Math.} \textbf{51} (1999), 816--834.

\bibitem {Hall2001}B. C. Hall, Harmonic analysis with respect to heat kernel
measure, \textit{Bull. Amer. Math. Soc. (N.S.)} \textbf{38} (2001), 43--78.

\bibitem {Halll2003}B. C. Hall, The Segal--Bargmann transform and the Gross ergodicity
theorem. \textit{In} Finite and infinite dimensional analysis in honor of
Leonard Gross (New Orleans, LA, 2001), 99--116, Contemp. Math., 317, Amer.
Math. Soc., Providence, RI, 2003.

\bibitem {HS}B. C. Hall and A. N. Sengupta, The Segal--Bargmann transform for
path-groups, \textit{J. Funct. Anal.} \textbf{152} (1998), 220--254.

\bibitem {HallExpository}B. C. Hall, The Segal--Bargmann transform for unitary
groups in the large-$N$ limit, preprint arXiv:1308.0615 [math.RT].

\bibitem{Ho}C.-W. Ho, The two-parameter free unitary Segal--Bargmann
transform and its Biane--Gross--Malliavin identification, \textit{J. Funct.
Anal.} \textbf{271} (2016), 3765--3817.

\bibitem {KempLargeN}T. Kemp, The large-$N$ limits of Brownian motions on
$GL(N)$, \textit{Int. Math. Res. Not.}, (2016), 4012--4057.

\bibitem {KNPS} T. Kemp, I. Nourdin, G. Peccati, and R. Speicher, Wigner chaos and the fourth moment.
\textit{Ann. Prob.} {\textbf 40} (2012), 1577--1635.

\bibitem{Lohmayer} R. Lohmayer, H. Neuberger, and T. Wettig, Possible large-%
\textit{N} transitions for complex Wilson loop matrices, \textit{J. High
Energy Phys.} 2008, no. 11, 053, 44 pp.

\bibitem {McKean} H. P. McKean, Stochastic Integrals.  Probability and Mathematical Statistics 5.
New York: Academic Press, 1969

\bibitem {MS}J. A. Mingo and R. Speicher, Free probability and random
matrices. Fields Institute Monographs, 35. Springer, New York; Fields
Institute for Research in Mathematical Sciences, Toronto, ON, 2017.

\bibitem {NS}A. Nica and R. Speicher, Lectures on the combinatorics of free
probability. London Mathematical Society Lecture Note Series, 335. Cambridge
University Press, Cambridge, 2006. xvi+417 pp.

\bibitem{Nowak} E. Gudowska-Nowak, R. A. Janik, J. Jurkiewicz, and M. A.
Nowak, Infinite products of large random matrices and matrix-valued
diffusion, \textit{Nuclear Phys. B} \textbf{670} (2003), 479--507.

\bibitem {RS}M. Reed and B. Simon, Methods of modern mathematical physics. I.
Functional analysis. Second edition. Academic Press, Inc. [Harcourt Brace
Jovanovich, Publishers], New York, 1980. xv+400 pp.

\bibitem {Singer} I. M. Singer, On the master field in two dimensions. \textit{In}
Functional analysis on the eve of the 21st century, Vol. 1 (New Brunswick, NJ, 1993),
volume 131 of \textit{ Progr. Math.}, 263--281.  Birkh\"auser Boston, Boston, MA, 1995.

\bibitem {TV}T. Tao and V. Vu, Random matrices: universality of ESDs and the
circular law. With an appendix by Manjunath Krishnapur. \textit{Ann. Probab.}
\textbf{38} (2010), 2023--2065.

\bibitem{Voiculescu}D. V. Voiculescu, Limit laws for random matrices and free products.
\textit{Invent. Math.} \textbf{104} (1991), 201--220.


\end{thebibliography}
\end{document}